\documentclass[11pt,letterpaper]{amsart}

\usepackage{amssymb}
\usepackage{amsfonts}
\usepackage{mathrsfs}
\usepackage{enumerate}
\usepackage{graphics}
\usepackage{latexsym}
\usepackage{ifthen}
\usepackage{url}
\usepackage[usenames]{color}

\newtheorem{theorem}{Theorem}[section]
\newtheorem{claim}[theorem]{Claim}
\newtheorem{lemma}[theorem]{Lemma}
\theoremstyle{definition}
\newtheorem{corollary}[theorem]{Corollary}

\newtheorem{definition}[theorem]{Definition}
\newtheorem{example}[theorem]{Example}

\newtheorem{proposition}[theorem]{Proposition}

\newcommand{\A}{\mathcal A}

\newcommand{\B}{\mathcal B}

\newcommand{\D}{\mathcal D}

\newcommand{\G}{\mathcal G}

\renewcommand{\H}{\mathcal H}

\newcommand{\N}{\mathbb N}

\newcommand{\M}{\mathcal M}

\newcommand{\R}{\mathcal R}

\newcommand{\Z}{\mathbb Z}
\newcommand{\con}{\otimes}

\begin{document}
\title{From automatic structures to automatic groups}

\author{Olga Kharlampovich \and Bakhadyr Khoussainov
  \and Alexei Miasnikov}

\maketitle
\markboth{(Version of \today)}{(Version of \today)}
\pagestyle{plain}

\begin{abstract}
In this paper we introduce the concept of  a Cayley graph automatic group (CGA group  or graph automatic group, for short) which generalizes the standard notion of an automatic group.  Like the usual automatic groups graph automatic ones enjoy many nice properties:  these group are invariant under the change of generators, they are closed under direct and free products, certain types of amalgamated products, and  finite extensions.  Furthermore, the Word Problem in graph automatic groups is decidable in quadratic time.   However, the class of graph automatic groups is much wider then the class of automatic groups. For example,  we prove that all finitely generated 2-nilpotent groups  and Baumslag-Solitar groups $B(1,n)$ are graph automatic, as well as many other metabelian  groups.
\end{abstract}

\tableofcontents

\section{Introduction}

Automata theory has unified many branches of algebra, logic, and computer science. These include group theory (e.g., automatic groups \cite{Thurston}, branch and self-similar groups \cite{branchgroups,Nek05}), the theory of automatic structures \cite{BG00,KN95,KM10,Ru08}), finite model theory, algorithms and decidability, decision problems in logic \cite{Buchi69,Rabin69}, model checking and verification \cite{Vardi}. \
In this paper we  use finite automata in representation of infinite
mathematical structures, emphasizing automata representations of groups via their Cayley graphs.

\smallskip

The idea of using automata to investigate algorithmic, algebraic and logical  aspects of mathematical structures goes back to the work of 
B\"uchi and Rabin \cite{Buchi69,Rabin69}. They established an intimate relationship between automata and the monadic second order (MSO) logic, where, to put it loosely, automata recognizability  is equivalent to definability in the MSO logic.  Through this relationship   B\" uchi proved that the MSO theory of one successor function on the set $\N$ is decidable \cite{Buchi69}.  Rabin used automata to prove that the MSO theory of two successors is decidable \cite{Rabin69}. The latter implies decidability of the first-order theories of many structures, for example:  linear orders, Boolean algebras, Presburger arithmetic, and term algebras \cite{Rabin69}.

\smallskip

In 1995 Khoussainov and Nerode, motivated by investigations in computable model theory
and the theory of feasible structures, used finite automata for representation of structures  \cite{KN95}, thus initiating the whole development of the theory of automatic structures (e.g. see \cite{KN95,BG00,Ru08,RuPhD}). Here a structure is called {\em automatic} if it is isomorphic to a structure whose domain and the basic operations and  relations are recognized by finite automata.  Automaticity implies  the following three fundamental properties of structures:
\begin{enumerate}
\item  The first order theory of every automatic structure is uniformly  decidable \cite{KN95} \cite{BG00};
\item  The class of automatic structures is closed under definability (with parameters) in the first order logic and in certain extensions of it \cite{BG00} \cite{KRS05} \cite{Kuske08};
\item  There is an automatic structure (a {\em universal automatic structure})  in which all other automatic structures  are first-order interpretable \cite{BG00}.
\end{enumerate}
 There are many natural examples of automatic structures: some fragments of the arithmetic,  such as $(\N ; +)$, state spaces of computer programs, the linear order of the rational numbers, the configuration spaces of Turing machines. However, not that many groups are automatic in this sense. In particular (see Section \ref{se:FAgroups}), a finitely generated  group is automatic (as a structure) if and only if it is virtually abelian.

\smallskip

In modern group theory there are already several ways to represent groups by finite automata.
One of these   is to consider finite automata with letter-by-letter outputs, known as Mealy automata. Every such automaton determines finitely many length preserving functions on the set of strings $X^\ast$ over the alphabet $X$ of the automaton.  If these functions are permutations then  they generate a group, called an {\em automata group}.  Automata groups enjoy some  nice algorithmic properties, for instance, decidability of the word problem. These groups are also a source of interesting examples. For instance,  the famous Grigorchuk group is an automata group. We refer to the book   \cite{Nek06} for  detail.

\smallskip

Another way to use finite automata in group representations comes from algorithmic and geometric group theory and topology.  Ideas of Thurston, Cannon, Gilman, Epstein and  Holt brought to the subject a new class of groups, termed {\em automatic groups}, and revolutionized computing with infinite groups (see the book \cite{Thurston} for details).   
The initial   motivation  for introducing automatic groups was 
two-fold: to understand the fundamental groups of compact 3-
manifolds and to approach  their  natural geometric structures  via the geometry and complexity of the optimal normal forms; and  to make them tractable for computing.  

Roughly,  a group $G$  generated by a finite set $X$ (with $X^{-1} = X$) is automatic if there exists a finite automata recognizable (i.e., rational or regular) subset $L$ of $X^\ast$ such that the natural mapping $u\rightarrow \bar{u}$ from $L$ into $G$  is bijective, and the right-multiplication by each of the generators from $X$ can be performed by a finite automata. 

This type of automaticity implies
some principal  "tameness" properties enjoyed by every automatic group $G$:
\begin{enumerate}
\item [(A)]  $G$ is  finitely presented.
\item [(B)] The Dehn function in $G$ is at most quadratic.
\item [(C)] There is a constant $k$ such that the words from $L$  (the normal forms) of elements in $G$ which are at most distance $1$ apart in the  Cayley graph $\Gamma(G,X)$ of $G$ are $k$-fellow travelers in $\Gamma(G,X)$. 
\end{enumerate}
Most importantly, as was designed at the outset, the word problem in automatic groups is  easily computable (the algorithmic complexity of  the conjugacy problem is unknown): 
\begin{enumerate}
\item [(D)] The word problem in a given automatic groups  is decidable in quadratic time.
\item [(E)] For any word $w \in X^\ast$ one can find in quadratic time it representative in $L$.
\end{enumerate}
Examples of automatic groups include hyperbolic groups, braid groups, mapping class groups, Coxeter groups, Artin groups of large type,  and many other groups. In addition, the class of automatic groups
is closed under  direct sums, finite extensions, finite index subgroups, free products, and some particular amalgamated free products. 
Yet many  classes of groups that possess nice representations and algorithmic properties fail to be automatic. Most strikingly, a finitely generated nilpotent group is automatic if and only if it is virtually abelian. 
To this end we quote  Farb \cite{Farb}: `` The fact that
nilpotent groups are not automatic is a bit surprising and annoying,
considering the fact that nilpotent groups are quite common and have
an easily solved word problem." The book \cite{Thurston} and a survey  by  Gersten \cite{Gersten99}  also raise a similar  concern, though they do not indicate what could be possible generalizations. 
In the view of the initial goals, nowadays  we know precisely from  Epstein-Thurston classification what are compact
geometrisable 3-manifolds whose fundamental groups are automatic \cite{Thurston}. The upshot
of this classification,  is that the fundamental group of a compact geometrisable 3-manifold $M$
is automatic if and only if none of the factors in the prime decomposition of $M$
is a closed manifold modeled on Nil or Sol. Thus, it turned out that the class of automatic groups is nice, but not sufficiently wide. In the geometric framework the quest for a suitable  generalization  comes inspired by the following "geometric" characterization of automatic groups as as those that have a regular set of normal forms $L \subseteq X^\ast$  satisfying ($C$).  Two main ideas are to replace the regular language $L$ with some more general language, and keep the fellow traveller property, perhaps, in a more general form.  Groups satisfying ($C$) with the formal language requirement of rationality weakened or eliminated entirely are called {\em combable}. In general combable groups are less amenable to computation than automatic groups. We refer to  the work of Bridson \cite{Bridson} for an account of the relation between combable and automatic groups.  On the other hand, a more relaxed fellow traveller property, called {\em asynchronous fellow travelers},  was introduced at the very beginning, see  the book \cite{Thurston}.  But Epstein and Holt \cite{Thurston} showed that the fundamental group of a closed Nil manifold is not even asynchronously automatic.  Finally, a geometric generalization of automaticity, that covers the  fundamental groups of all compact 3-manifolds which satisfy the geometrization conjecture, was given by  Bridson and Gilman in \cite{BG}. 
However, as far as we know, these  geometrically natural generalizations loose the nice algorithmic properties mentioned above.

In fact, from the algorithmic standpoint, the  properties ($A$) and ($B$) can be viewed as unnecessary restrictions, depriving automaticity for  a wide variety of otherwise algorithmically nice groups such  as nilpotent or metabelian groups.
\vspace{2mm}

 In this paper we propose a natural generalization of automatic groups and  introduce the class of  Cayley graph automatic groups. A finitely generated group $G$ is called {\em Cayley graph automatic}, or {\em graph automatic or CGA} for short,  if it satisfies the definition of an automatic group as above, provided the condition that the alphabet $X$ is a set of generators of $G$, is removed.  Equivalently, a finitely generated group is {\em graph automatic} if its Cayley graph is an automatic structure in the sense of Khoussainov and Nerode.  The former definition immediately implies that the standard automatic groups are graph automatic.   However, there are many examples of graph automatic groups that are not automatic. These include Heisenberg groups, Baumslag-Solitar groups $BS(1,n)$, arbitrary finitely generated groups of nilpotency class two, and  some nilpotent groups of higher class, such as unitriangular groups $UT(n,\mathbb{Z})$, as well as many metabelian groups and solvable groups of higher class, like $T(n,\mathbb{Z})$. Moreover, we do not have the restrictions ($A$) and ($B$)  any more. As in the case of automatic groups the class of graph automatic groups is closed under free products, direct sums, finite extensions, wreath-products, 
 and certain types of amalgamated products,  etc.  This shows that the class of graph automatic groups, indeed, addresses some concerns mentioned above, but whether the class is good enough remains to be seen. Firstly, we do not know if  every 
finitely generated nilpotent group is graph automatic or not, in particular, the question if a finitely generated free nilpotent group of class 3 
is graph automatic is still open. Likewise, we do not know any geometric condition that would give a characterization of graph automatic groups similar to property ($C$). On a positive side though there is a crucial algorithmic result stating that again the word problem for graph automatic groups is decidable in quadratic time, so property ($D$) is preserved. Moreover, there is a new "logical" condition that gives a powerful test  to check if  a given group is graph automatic.  Namely, the group $G$ is graph automatic if and only if its Cayley graph $\Gamma(G,X)$ is first-order interpretable in an automatic structure or, equivalently, is interpretable in a fixed {\em universal} automatic structure (see above and also Section \ref{se:universal} for definitions and examples).  In addition,  there is a natural notion of a {\em graph biautomatic} group, which generalizes the standard class of biautomatic groups, with similar algorithmic properties. For instance, the conjugacy problem in graph biautomatic groups is decidable. In this case the proofs are simpler and more straightforward than in the classical one.
It seems it might be a chance to address the old problem whether automaticity implies biautomaticity in this new setting, which might shed some light on the old problem itself, but presently this is a pure speculation. 

\smallskip

On a philosophical note we would like to mention that there is a large overlap in ideas and proof methods  employed in the study of automatic structures and automatic groups of various types. In this paper we also aim to present the ideas and the proof methods in a unified form, which makes them  more available for use in both areas.  

\section{Finite automata}

We start with the basic definitions from finite automata theory.
Let  $\Sigma$ be  a finite  alphabet. The set of all finite strings over
$\Sigma$ is denoted by $\Sigma^\star$.  The variables $u$, $v$, $w$ represent strings. The empty string is  $\lambda$. The length of a string $u$ is denoted by $|u|$. For a set $X$, $P(X)$ is the set of all subsets of $X$. The cardinality of $X$ is denoted by $|X|$.

\begin{definition}
 A {\bf nondeterministic finite automaton} (NFA for short)
 over $\Sigma$ is a tuple $(S,  I, T, F)$, where $S$ is the set of {\bf states}, $I\subseteq S$ is the set of initial states,  $T$ is the {\bf transition function} $T:S\times \Sigma \rightarrow P(S)$,  and
  $F\subseteq S $ is the set of {\bf accepting states}.
We use the letter $\M$ possibly with indices to denote NFA.
\end{definition}
One can visualize an NFA  $\M$ as a labeled graph called the transition diagram of the automaton. The states of the NFA represent the  vertices  of the graph. We put a directed edge from state $s$ to state $q$ and label it with $\sigma$  if $q\in T(s,\sigma)$. These are called {\bf $\sigma$-transitions}.

\smallskip

Let $\M$ be an NFA. A {\bf run} of $\M$ on the string $w=\sigma_1\sigma_2\ldots \sigma_n$ is a sequence of states \ $
s_1, s_2, \ldots,s_n, s_{n+1} $
such that  $s_1\in I$ and $s_{i+1}\in T(s_i, \sigma_i)$ for all $i=1,\ldots, n$. The automaton might have more than one run on  the  string $w$. These runs of the automaton can be viewed as paths in the transition diagram labeled by $w$.

\begin{definition}
The automaton $\M$ accepts the string  $w=\sigma_1\sigma_2\ldots \sigma_n$
if $\M$ has a run  \ $s_1, s_2, \ldots,s_n, s_{n+1} $ on $w$
such that  $s_{n+1} \in F$.  \
The {\bf language accepted by $\M$}, denoted by $L(\M)$, is the following language:
$$
\{w \mid \  \mbox{the automaton $\M$ accepts $w$}\}.
$$
A language $L\subseteq \Sigma^\star$ is {\bf FA recognizable} if there exists an NFA $\M$ such that $L=L(\M)$.
\end{definition}

It is well-known that the set of all NFA recognizable languages in $\Sigma^\star$ forms a Boolean algebra under  the set-theoretic operations of union, intersection, and complementation; and every NFA recognizable language is also recognizable by a deterministic finite automata. By well-known Klenee's theorem the class of FA recognizable languages coincised with the class of regular languages. Therefore, often we refer to FA recognizable languages also as regular languages.

\smallskip

We now introduce the notion of automata that recognizes  relations over the set $\Sigma^\star$.
This is done through the following definitions and notation. \ We write $\Sigma_{\diamond}$ for $\Sigma \cup \{\Diamond\}$ where $\Diamond\not \in \Sigma$.
The {\em convolution of a tuple} $(w_1,\cdots,w_n) \in
\Sigma^{\star n}$ is the string
$$
\con (w_1,\cdots,w_n)
$$
of length $\max_i|w_i|$ over alphabet $(\Sigma_{\diamond})^n$ defined as follows.  The  $k$'th symbol of the string is $(\sigma_1,\ldots,\sigma_n)$ where $\sigma_i$ is the
$k$'th symbol of $w_i$ if $k \leq |w_i|$ and $\diamond$ otherwise. For instance, for $w_1=aabaaab$, $w_2=bbabbabbb$, and $w_3=aab$, we have
$$
\con (w_1, w_2, w_3)=
\begin{pmatrix}
a& a& b& a & a& a & b & \diamond & \diamond  \\
b & b & a & b & b & a & b & b & b\\
a& a & b & \diamond & \diamond  & \diamond & \diamond & \diamond & \diamond
\end{pmatrix}
$$

\begin{definition}
The {\em convolution of a relation} $R \subset \Sigma^{\star n}$ is the relation
$\con R \subset (\Sigma_{\diamond})^{n\star}$ formed as the set of convolutions
of all the tuples in $R$, i.e.,
$$
\con R = \{\con (w_1,\ldots, w_n) \mid (w_1,\ldots, w_n) \in R\}.
$$
\end{definition}

\noindent
The convolution operation codes up relations in $\Sigma^{\star n}$ into usual languages but over a special  alphabet.
This allows one to define finite automata recognizable relations in the standard way.

\begin{definition} \label{dfn:n-tape}   An $n$--ary
  relation $R \subset \Sigma^{{\star}n}$ is {\bf  FA recognizable}  if its convolution $\con R$ is
  recognizable by a finite automaton in the alphabet $(\Sigma_{\diamond})^n$. \end{definition}

Intuitively, a finite automaton recognizing an $n$-ary relation $R \subset \Sigma^{{\star}n}$ can be viewed as  a finite automaton with $n$ heads. All heads read distinct tapes and make  simultaneous transitions. Therefore, finite automata over the alphabet $(\Sigma_{\diamond})^n$ are also called  a {\em  synchronous $n$--tape  automaton} on $\Sigma$. FA recognizable relations in $\Sigma^{{\star}n}$ are also called  {\em regular relations}.

\begin{example}\label{Ex:orders}
Here are several examples of FA recognizable relations over  $\Sigma$.
These are examples of linear orders on $\Sigma^\star$ that are recognized by finite automata:
\begin{itemize}
\item  The lexicographic order on strings:  $\leq_{lex}=\{(x,y) \mid x, y \in \Sigma^\star$ and $x$ is lexicographically less than $y$ or $x=y\}$.
\item
The prefix order on strings: $\leq_{pref}=\{(x,y) \mid x, y \in \Sigma^\star$ and $x$ is a prefix of $y\}$.
\item  The length-lexicographic order on strings:   $\leq_{llex}=\{(x,y) \mid x, y \in \Sigma^\star$ and either $|x|<|y|$ or  ($|x|=|y|$ and $x\leq_{lex} y)\}$.
\end{itemize}
 \end{example}

We note that often in the text we identify the convoluted word $\con(w_1,\ldots, w_n)$ with the tuple $(w_1,\ldots, w_n)$. 
This will be clear from the context. 

\section{Automatic structures}

In this section we introduce automatic structures, give several examples, and provide some known results on automatic structures.  By a {\bf structure} $\A$ we  mean a tuple
$$(A; P_0^{n_0}, \ldots, P_k^{n_k}, f_0^{m_0},\ldots,  f_t^{m_t}),$$
 where the set $A$ is the domain of the structure $\A$, each $P_i^{n_i}$ is a relation of arity
$n_i$ on $A$, and each $f_j^{m_j}$ is a total operation of arity $m_j$ on $A$.
These relations and operations are often called atomic.
The structure $\A$ is {\bf relational} if it contains no operations. Every structure $\A$ can be transformed into a relational structure. This is done by replacing each atomic operation $f_j^{m_j}:A^{m_j} \rightarrow A$ with its  graph:
$$
Graph (f_j^{m_j})=\{(a_1, \ldots, a_{m_j},a) \mid f_j^{m_j}(a_1, \ldots, a_{m_j})=a\}.
$$
The sequence  of symbols $P_0^{n_0}, \ldots, P_k^{n_k}, f_0^{m_0},\ldots,  f_t^{m_t}$ is called a signature of the structure. Below is the key  definition of this paper:

\begin{definition} \label{CentralDefinition}
The structure  $\A=(A; P_0^{n_0}, \ldots, P_k^{n_k}, f_0^{m_0},\ldots,  f_t^{m_t})$ is called  {\bf automatic} if all the domain  $A$, the predicates $P_0^{n_0}$, $\ldots$, $P_1^{n_k}$, and the graphs of operations $Graph (f_0^{m_0})$, $\ldots$,$Graph(f_t^{m_t})$ are FA recognizable.
\end{definition}

Here are some examples of automatic structures:
\begin{example}
 The structure  $(1^\star; S,  \leq))$, where $S(1^n)=1^{n+1}$ and $1^n\leq 1^m$ iff $n\leq m$ for $n,m \in \omega$.
\end{example}

\begin{example}
 The structure $(\{0,1\}^\star; \leq_{lex}, \leq_{pref}, \leq_{llex})$, where the orders are defined in Example \ref{Ex:orders}.
\end{example}

\begin{example} \label{Ex:Presburger}
The structure $(Base_k; Add_k)$, where $Base_k=\{0,1,\ldots, k-1\}^\star \cdot \{1,\ldots, k-1\}$. In this example each word $w=x_0\ldots x_n \in Base_k$  is identified with the number
$$
val_k(w)=\sum_{i=0}^{n} x_i k^i.
$$
This gives the least significant digit first base-$k$-representation of natural numbers. The predicate $Add_k$ is the graph of the $k$-base  addition of natural numbers, that is $Add_k=\{(u,v,w) \mid val_k(u)+val_k(v)=val_k(w)\}$. This structure is isomorphic to
the natural numbers with addition $\mathcal P= \langle \mathbb{N}, +\rangle$ known as Presburger arithmetic.
\end{example}

\smallskip

\noindent
Let $\A$ be a structure. The {\bf isomorphism type} of the structure is the class of all structures isomorphic to it. We identify structures up to isomorphism. Therefore, we are interested in those structures whose isomorphism types contain automatic structures.

\begin{definition}
A structure $\B$ is {\bf automata presentable} if there exists an automatic structure $\A$ isomorphic to $\B$. In this case $\A$ is called an {\bf  automatic presentation} of $\B$.
\end{definition}

We would like to give several comments about this definition. The first is that an automatic presentation  $\A$ of a structure $\B$ can be viewed  as a finite sequence of automata
representing the domain, the atomic relations, and operations of the structure. The sequence is finite.
Hence, automatic presentations are just finite objects that describe the structure. The second is that if a structure $\B$ has an automatic presentation, then $\B$ has infinitely many automatic presentations. Finally,  in order to show that $\B$ has an automatic presentation, one needs to find an automatic structure $\A$ isomorphic to $\B$. Thus, to show that $\B$ does not possess an automatic presenation one needs to prove  that for all automatic presentations $\A$ all bijective mappings $f:\B \rightarrow \A$ fail to establish
an  isomorphism. In a logical formalism the defnition of automaticity  is a $\Sigma_1^1$-definition in the language of arithmetic.

\smallskip

Since we are mostly interested in the isomorphism types of the structures, we often abuse our definitions and refer to automata presentable structures as automatic structures. Below we give some examples
of automatic (automata presentable) structures.

\begin{itemize}
\item  The Boolean algebra $\B_{\omega}$ of finite and co-finite subsets of $\N$. To show that $\B_{\omega}$ is automata presentable,  we code elements of $\B_{\omega}$ as finite binary strings in a natural way. For example, the string $01101101$ represents the infinite set $\{1,2,4,5,7,8,9,10,11,\ldots\}$ and the string $0110110$ represents the finite set $\{1,2,4,5\}$.  Under this coding,   $\B_{\omega}$ is  automata presentable.
\item The additive group $(\Z, +)$ is automata presentable.
\item  Finitely generated abelian groups are automata presentable. This follows from the fact that
such groups are finite direct sums of $(\Z, +)$ and finite abelian groups.
\item  Small ordinals of the form  $\omega^n$, where $n\in \N$.
\end{itemize}

For a given formula $\phi(\bar{x}_1, \ldots, \bar{x}_k)$,
we set $\phi(\A)$ be all tuples $(\bar{a}_1, \ldots, \bar{a}_k)$ in structure $\A$ that satisfy $\A$.
We now give the following definition that will be used in this paper quite often:

\begin{definition}  \label{Dfn:Interp}
A structure $\B=(B; R_1,\ldots, R_m)$ is {\bf interpretable in structure} $\A=(A,S_1,\ldots, S_n)$ if there are 
formulas 
$$
\D(\bar{x}),  \ \phi_1(\bar{x}_1,\ldots, \bar{x}_{k_1}), \ldots,  \ \phi_m(\bar{x}_1,\ldots, \bar{x}_{k_m})
$$
of the first order logic such that:
\begin{enumerate}
\item all tuples $\bar{x}$, $\bar{x}_1$, $\ldots$, $\bar{x}_{k_n}$ of variables have the same length, and 
\item The structure $(D(\A); \phi_1(\A), \ldots, \phi_m(\A))$
is isomorphic to $\B$.
\end{enumerate}
\end{definition}

The following is a foundational theorem in the study of automatic structures. The proof of the theorem follows from closure properties of finite automata under set-theoretic Boolean operations, the projection operation, and decidability of the emptiness problem for automata. Recall that the emptiness problem asks if there exists an algorithm to check if the language $L(\M)$ of  a given finite automaton $\M$ is empty or not. 

\begin{theorem}[The Definability and Decidability Theorem] \cite{BG00} \cite{KN95} \cite{KRS05}
\label{FDT}
\begin{enumerate}
\item  There is an algorithm that, given an automatic presentation of any structure $\A$  and a first-order formula $\varphi(x_1, \ldots, x_n)$, produces an automaton recognizing those tuples $(a_1, \ldots, a_n)$ that make the formula true in $\A$.
\item If a structure $\A$ is first-order interpretble in an automatic structure $\B$ then $\A$ has an automatic presentation.
\item The first-order theory of every automatic structure is decidable.\qed
\end{enumerate}
\end{theorem}

Note that there are several generalizations of this theorem to logics extending the first order logic. One important generalization is the following.
To the first order logic $FO$ add the following two quantifiers: $\exists^{\omega}$ (there exists infinitely many) and $\exists ^{n,m}$ (there exists $m$ many modulo $n$). Denote the resulting logic as $FO+\exists^{\omega}+\exists^{n,m}$. The theorem above can be extended to this  extended logic \cite{RuPhD} and other automata presentable structures \cite{Kuske08}.

\smallskip

We will be using the theorem above without explicit  references. For instance, the Presburger arithmetic ${\mathcal P} = \langle \mathbb{N}, +\rangle$ is clearly an automatic structure by Example \ref{Ex:Presburger}. Hence any structure definable in ${\mathcal P}$ is automatic. We will use this observation  in some of our arguments.


Furthermore, there are some {\em universal} automatic structures, i.e., automatic structures $\A$ such that a structure $\B$ has an automatic presentation  if and only if it is first-order interpretable  in $\A$.
Consider the following two structures:

\begin{example} 
$$\N_2 = (\mathbb{N}; +,|_2),$$
where $+$ is the standard addition and $x|_2y \Leftrightarrow x = 2^k \ \& \  y = x\cdot z$ for some $k, z \in \mathbb{N}$.
\end{example}

\begin{example}  \label{ex:free-monoid}
$${\mathcal M} = (\Sigma^\ast; R_a(x,y), x\preceq y, el(x,y))_{ a \in \Sigma},$$
 where $\Sigma$ is a finite alphabet with $|\Sigma| \geq 2$, $R_a(x,y) \leftrightarrow y = xa$, $x\preceq y \leftrightarrow $ $x$ is a prefix of $y$, $el(x,y) \leftrightarrow |x| = |y|$. \qed
\end{example}

The following theorem gives a pure  model theoretic characterization of automatically presentatble structures.
\begin{theorem} \cite{BG00} \label{th:universal}
The structures $\N_2$  and ${\mathcal M}$ are universal automatic structures. In particular, $\N_2$ and ${\mathcal M}$ are interpretable in each other. 
\end{theorem}

\section{Cayley graphs}

In the next section we will introduce automaticity into groups through their Cayley graphs.
This section recalls the definition of Cayley graphs and some of their basic properties.

\smallskip

Let $G$ be an infinite group generated by a finite set $X$. There exists a natural onto map from $X^\star$
into $G$ mapping the words $v$ into the group elements $\bar{v}$. {\bf The word problem} for $G$ (with respect to $X$)
is the following set
$$
W(G,X)=\{(u,v)\mid u,v \in X^{\star}\ \& \ \bar{u}=\bar{v} \ \mbox{in group G}\}.
$$
The word problem for $G$ is decidable if there exists an algorithm that given two words $u, v\in X^{\star}$ decides
if $\bar{u}=\bar{v}$. It is not hard to see that decidability of $W(G,X)$ does not depend on finite set of generators for $G$.

\smallskip

The group $G$ and the finite set $X$ of generators determine the following graph,
called a {\em Cayley graph of $G$}, and denoted by $\Gamma(G,X)$.
The vertices of the graph are the elements of the group. For each vertex $g$ we put a directed edge from $g$ to $gx$, where $x\in X$, and label the edge by $x$.  Thus, $\Gamma(G,X)$ is a labeled directed graph.

  We view a labeled directed graph $\Gamma=(V,E)$ with the labels of the graph  from a finite set
$\Sigma=\{\sigma_1, \ldots, \sigma_n\}$ as the following structure:
$$
(V, E_{\sigma_1}, \ldots, E_{\sigma_n}),
$$
where $E_{\sigma}=\{(x,y) \mid (x,y)\in E$ and the label of $(x,y)$ is $\sigma\}$ for $\sigma\in \Sigma$.
\smallskip

For the next lemma we need one definition from the theory of computable structures.  We say that a labeled directed graph $(V, E)$ is {\bf computable} if its vertex set $V$ and the labeled edge sets $E_x$, $x\in X$,  are all
computable subsets of  $\Sigma^\star$, where $\Sigma$ is a finite alphabet.  We refer to the set
$\Sigma^\star$ as {\em the domain of discourse}.  Thus, computability of
the graph $(V, E)$ states  that
\begin{enumerate}
\item There is an algorithm that given a vertex  $u$ from the domain of discourse,
decides if $u\in V$, and
\item  There is an algorithm that,  given any two  vertices $v$ and $w$ from $V$ and a label $x\in X$,
decides if there exists an edge from $v$ to $w$ labeled by $x$.
\end{enumerate}

Here are some basic properties of the Cayley graph $\Gamma(G,X)$.

\begin{lemma}
The Cayley graph $\Gamma(G,X)$ satisfies the following properties:
\begin{enumerate}
\item The graph is strongly connected, that is between any two vertices of the graph there is a path connecting one to another.
\item The out-degree and in-degree of each node is $|X|$.
\item The graph is transitive, that is, for any two vertices $g_1$ and $g_2$ of the graph there exists an automorphism $\alpha$ of $\Gamma(G,X)$ such that $\alpha(g_1)=g_2$.
\item The group of (label respecting)  autmorphisms of $\Gamma(G,X)$ is isomorphic to $G$.
\item The graph $\Gamma(G,X)$ is computable if and only if the word problem in $G$ is decidable.
\end{enumerate}
\end{lemma}

\begin{proof} The first four parts of the lemma  are standard. The last part of the lemma needs an explanation. Assume that the word problem $W(G,X)$ in $G$  is decidable. Then  there exists an algorithm that, given any two words $w$ and $v$ over $X$, decides if $v=w$ in the group $G$. Now  we construct the graph $\Gamma(G,X)$ as follows.  The vertex set  $V$ of the graph consists of all words $v\in X^\star$ such that any word $w$ that is equal to $v$ in
$G$ is length-lexicographically larger than or equal to $v$. Clearly, this set $V$ of vertices is computable. Since the word problem is decidable in $G$, we can use the algorithm for the word problem to decide  if there exists an edge from $v_1$ to $v_2$ labeled by  $x\in X$. This shows that $\Gamma(G, X)$ is a computable graph.  Assume that the Calyey graph $\Gamma(G,X)$ is computable.  Then given any two words $w_1$  and $w_2$ in $X^{\star}$ one can effectively find two vertices $v_1$ and $v_2$ that represents $w_1$ and $w_2$ in the graph, respectively.  Then $w_1=w_2$ in the group $G$ if and only if $v_1=v_2$. Hence the word problem in $G$ is decidable.
\end{proof}

\section{$\aleph_1$-categoricity}

We prove one model-theoretic property of Cayley graphs that has an algorithmic implication.
This will have a direct relation with automaticity. 
We start with one important definition from model theory \cite{Ho93}. A complete theory $T$ in the first order logic  is  called {\em $\aleph_1$-categorical}
if all models of $T$ of cardinality $\aleph_1$ are isomorphic.
The lemma below shows  that the theory of every infinite Cayley graph is $\aleph_1$-categorical.
This result has two consequences. One is that decidability of the word problem for $G$ is equivalent to
decidability of the first order theory of its Cayley graph. The other will have a direct relation with automaticity that will be seen in a later 
sections.

\smallskip

We start with a lemma (interesting in its own right)  that is true for all locally finite labeled directed graphs. In particular,
the lemma can be applied to Cayley graphs. Recall that a directed graph is {\bf locally finite} if every vertex
of the graph has finitely many in-going and out-going edges. So, let $\Gamma$ be a locally finite labeled and directed graph. For vertices $x,y$ of the graph $\Gamma$, we set $d(x,y)$ be the length of the shortest path from $x$ to $y$. Then the $n$-ball around a vertex $x$ is the set
$$
B_n(x)=\{y \mid d(x,y)\leq n\}.
$$
\begin{lemma}
Let $\Gamma_1$ and $\Gamma_2$ be locally finite labeled connected and directed  graphs.
Assume that $a$ and $b$ are vertices of $\Gamma_1$ and $\Gamma_2$, respectively, such that for
all $n\in \N$ there is an isomorphism from $B_n(a)$  to $B_n(b)$ sending $a$ to $b$. Then there exists an isomorphism $\alpha:\Gamma_1 \rightarrow \Gamma_2$ such that
$\alpha(a)=b$.
\end{lemma}

\begin{proof}
Each of the $n$-balls $B_n(a)$ and $B_n(b)$ is a finite set. There are finitely many isomorphisms from $B_n(a)$ into $B_n(b)$ that send $a$ to $b$. Denote this set by $I_n$. By assumption, $I_n\neq \emptyset$ for all $n\in \N$. Each such isomorphism $\alpha\in I_n$ induces an isomorphism $\alpha^\prime$ from $B_{n-1}(a)$ to $B_{n-1}(b)$; so, $\alpha' \in I_{n-1}$. It is easy to see that the collection of all finite isomorphisms from $B_n(a)$ into $B_n(b)$, where $n\in \N$, determines a finitely branching infinite tree (edges connect the isomorphisms $\alpha$ and $\alpha^\prime$).  Now apply K\"onig's lemma to select an infinite path $P$ along the tree. This path $P$ determines an isomorphism from $\Gamma_1$ to $\Gamma_2$
that sends $a$ to $b$.
\end{proof}

\begin{lemma}
Let $G$ be a group generated by a finite set $X$. The theory of the Cayley graph $\Gamma(G,X)$ is $\aleph_1$-categorical.
\end{lemma}

\begin{proof}
Fix any element $g$ of the Cayley graph $\Gamma(G,X)$. Consider the $n$-ball $B_n(g)$. Since
$\Gamma(G,X)$ is transitive $B_n(g')$ is isomorphic to $B_n(g)$ for all $g'\in G$. The theory $T(G,X)$ of the graph contains the following sentences:
\begin{enumerate}
\item  The sentence $\Phi_{n,m}$ stating that there are $m$ distinct elements $x$ such that $B_n(x)$ is isomorphic to $B_n(g)$, where $m,n \in \N$. Note that this is an infinite set of axioms.
\item The sentence $\Phi_n$ stating that for all $x$ the $n$-ball $B_n(x)$ around $x$
is isomorphic to $B_n(g)$, where $n\in \N$.
\end{enumerate}

The theory $T(G,X)$ described has a model which is the Cayley graph $\Gamma(G,X)$ since
the graph $\Gamma(G,X)$ satisfies all the sentences of the theory. Our goal is to show that any two models
$\A$ and $\B$ of this theory are isomorphic in case $\A$ and $\B$ have cardinality $\aleph_1$.

\smallskip

We note that each $\A$ and $\B$ is a labeled directed locally finite graph. As graphs they consist of strongly connected components. We refer to them simply as components. Note that each component in the graph $\A$,
and hence in $\B$, is  countable since the in-degree and out-degree of every element in $\A$ is
$|X|$. This implies that  both $\A$ and $\B$ are disjoint unions of their components, where  the cardinality of the union is $\aleph_1$.

\smallskip

Now, we show that any two components of $\A$ and $\B$ are isomorphic.
Indeed, take two elements $a\in \A$ and $b\in \B$, respectively. By the axioms of $T(G,X)$, for each $n\in \N$
there is an isomorphism from $B_n(a)$ to $B_n(b)$ that maps $a$ to $b$.  Apply the lemma
above to build an isomorphism from the component of $a$ onto the component of $b$. This shows that all components of  the graphs $\A$ and $\B$ are pairwise isomorphic. Therefore, we match the components of $\A$ with components of $\B$, and  build an isomorphism from $\A$ to $\B$. Thus, $T(G,X)$ is an $\aleph_1$-categorical theory.
\end{proof}

It is worth to note that the lemma stays true if we remove the labels from the edges of the Cayley graph $\Gamma(G,X)$. Namely, let $\Gamma_u(G,X)$ be the directed graph obtained from $\Gamma(G,X)$ be
removing the labels from all the edges. Then the theory of the unlabeled graph $\Gamma_u(G,X)$ is $\aleph_1$-categorical.

\smallskip

The lemma above allows us to address decidability of the word problem for the group $G$ in terms of decidability of the  theory $T(G,X)$:

\begin{theorem}\label{CayleyDecidable}
The word problem $W(G,X)$  in $G$ is decidable if and only if the theory $T(G,X)$ of the Cayley graph $\Gamma(G,X)$ is decidable.
\end{theorem}

\begin{proof}
Assume that the word problem in $G$ is decidable. Our goal is to show that
the theory $T(G,X)$ is also decidable.
It is clear that $T(G,X)$ is effectively axiomtizable by the sentences $\Phi_{n,m}$ and $\Phi_n$ as follows from  the proof of the lemma
above. It is known that every $\aleph_1$-categorical theory $T$ without finite models is complete, that is,
for any sentence $\phi$ either $\phi$ belongs to $T$ or $\neg \phi$ belongs to $T$ \cite{Ho93}. From
the lemma above, we  conclude that $T(G,X)$ is a complete first order theory.
 Since $T(G,X)$ is complete, for every $\phi$ either $\phi$ or $\neg \phi$  is deducible from  the axioms  $\Phi_{n,m}$ and $\Phi_n$. This implies decidability of $T(G,X)$.

\smallskip

Assume that the theory $T(G,X)$  is decidable. Clearly,
$\Gamma(G,X)$ is a model of $T$.  Now we use the result of Harrington  (and independently Khisamiev) that states the following. If $T$ is $\aleph_1$-categorical decidable theory then all of its countable models are computable \cite{Ha74} \cite{Kh74}. We conclude that $\Gamma(G,X)$ is also a computable model of $T(G,X)$\footnote{In fact, one can effectively build the graph $\Gamma(G,X)$ without referencing Harrington and Khisamiev's theorems. The reader can construct $\Gamma(G,X)$ as an exercise.}. This proves the theorem.
\end{proof}

\section{Cayley graph automatic groups: definitions and examples}

In this section we introduce labeled automatic graphs and present several examples.   Let $\Gamma=(V,E)$ be a labeled directed graph. The labels of the graph are from a finite set
$\Sigma=\{\sigma_1, \ldots, \sigma_n\}$.

\begin{definition}
We view the graph $\Gamma$ as the following structure:
$$
(V, E_{\sigma_1}, \ldots, E_{\sigma_n}),
$$
where $E_{\sigma}=\{(x,y) \mid (x,y)\in E$ and the label of $(x,y)$ is $\sigma\}$ for $\sigma\in \Sigma$. We say that the graph $\Gamma$ is {\bf automatic} if the structure
$(V, E_{\sigma_1}, \ldots, E_{\sigma_n})$ is automatic.
\end{definition}

Here are examples of automatic graphs.
\begin{example}
Let $T$ be a Turing machine. The configuration space of $T$ is the graph $(Conf(T), E_T)$, where:
\begin{enumerate}
\item The set  $Conf(T)$ is the set of all configurations of $T$, and
\item The set $E_T$ of edges consists of all pairs $(c_1, c_2)$ of configurations such that
$T$ has an instruction that transforms $c_1$ to $c_2$.
\end{enumerate}
The structure $(Conf(T), E_T)$ is clearly an automatic directed graph since the transitions $(c_1, c_2) \in E_T$ can be detected by finite automata.
\end{example}

The next example shows the the $n$-dimensional grid is also an automatic graph.

\begin{example} \label{Ex:Z^n}
Consider $\Z^n$ as a labeled graph, where the labels are $e_1$, $\ldots$, $e_n$.
Identify each $e_i$ with the vector $(0,\dots, 0, 1, 0,\ldots, 0)$, whose all components are $0$ except at position $i$. For any two vectors $v$ and $w$ in $\Z^n$, put an edge from $v$ to $w$ and label it with $e_i$ if $v+e_i=w$. We represent each vector $v\in \Z^n$ as an $n$-tuple $(x_1, \ldots, x_n)$ of integers each written in a binary (or decimal) notation. Under this coding,
the edge relation
$$
E_i=\{(v,w) \mid v+e_i=w\}
$$
is FA recognizable. Hence, the labeled  graph $\Z^n$ is automatic.
\end{example}

\noindent

The next is a central definition of this paper that introduces automaticity for finitely generated groups.

\begin{definition} \label{Dfn:Central}
Let $G$ be a group generated by a finite set $X$ of generators. We say that $G$ is {\bf Cayley graph automatic} if the graph $\Gamma(G,X)$ is an automatic graph. We often refer to Cayley graph automatic groups as either {\bf graph automatic groups} or {\bf CGA groups}.
\end{definition}

\noindent
Here are several examples.

\begin{example} \label{Ex:Abelian}
Consider a  finitely generated abelian group $G$. The group $G$ can be written as $\Z^n \bigoplus A$, where $A$  is a finite abelian group and $n\in \N$. The group $G$ is generated by $A$ and
 the vectors $e_1$, $\ldots$, $e_n$ in $\Z^n$.  Using Example \ref{Ex:Z^n} and the fact that $A$ is finite, it is easy to show that the group $G$ is graph automatic.
 \end{example}

\begin{example} \label{Ex:Heisenberg}
The Heisenberg group  $\H_3(\Z)$ consists of $3\times 3$ matrices $X$ over $\Z$ of the following type:
$$
X=\left(
\begin{array}{ccc}
1 & a & b\\
0 & 1 & c\\
0 & 0 & 1
\end{array}
\right).
$$
The group has $3$ generators which are
$$
A=\left(
\begin{array}{ccc}
1 & 1 & 0\\
0 & 1 & 0\\
0 & 0 & 1
\end{array}
\right), \
B=\left(
\begin{array}{ccc}
1 & 0 & 1\\
0 & 1 & 0\\
0 & 0 & 1
\end{array}
\right),
 \  \mbox{and} \
 C=\left(
\begin{array}{ccc}
1 & 0 & 0\\
0 & 1 & 1\\
0 & 0 & 1
\end{array}
\right).
$$
We can represent the matrix $X$ as the convoluted word $\con (a,b,c)$, where $a$, $b$ and $c$ are written in binary. The multiplication of $X$ by each of these generators can easily be recognized by finite automaton. Indeed, the three automata that recognize the multiplication by $A$, $B$, and $C$ accept all the strings of the form
$\con (\con (a,b,c), \con (1+a, b,c))$,  $\con (\con (a,b,c), \con (a, 1+b,c))$, and
$\con (\con (a,b,c), \con (a, b,1+c))$,  respectively. Thus, $\H_3(\Z)$ is a graph automatic group.
\end{example}

\begin{example} \label{Ex:n-Heidelberg}
The example above can clearly be generalized to Heisenberg groups $\H_n(\Z)$ consisting of all $n\times n$ matrices over $\Z$ which have entries  $1$ at the diagonal and whose all other entries apart from first row or the last column are equal to $0$.
\end{example}

Now we mention some properties of graph automatic groups that follow directly from the definitions.
We start with the following easy lemma.

\begin{lemma} \label{L:Product}
Let $G$  be a graph automatic group over a generating set $X$. Then
for a given word $y \in (X \cup X^{-1})^\ast$ there exists a finite
automaton $\M_{y}$ which accepts all the pairs  $u, v\in \Gamma(G,X)$ with $v=uy$ and nothing else.
\end{lemma}
\begin{proof}
Since $\Gamma(G,X)$ is automatic,  for every $x \in X$ there exists an automaton $\M_x$  such that for
all $u,v \in \Gamma(G,X)$, the automaton $\M_x$ detects if $v=u\cdot x$.
Now one  can use the automata $\M_x$, $x\in X$,  to build
a finite automaton $M_{y}$ that recognizes all $u,v  \in \Gamma(G,X)$ such that $v=uy$. This can be done through
Theorem \ref{FDT}. Indeed, there exists a formula $\phi(w,v)$ in the language of the Cayley graph  $\Gamma(G,X)$ with free variables $u,v$ such that $\phi(u,v)$ holds in $\Gamma(G,X)$  if and only if $v=uy$ in $G$. So the binary predicate defined by $\phi(u,v)$ is FA recognizable 
in $\Gamma(G,X)$. Hence we can build the desired automaton $\M_y$.
\end{proof}

The theorem below shows that the definition of graph automaticity is independent on the generator sets. The proof is much simpler than 
the proof of the similar results for standard automatic groups \cite{Thurston}. 

\begin{theorem}
If  $G$ is a graph automatic group with respect to a generating  set $X$ then $G$ is Cayley graph
automatic with respect to all finite generating  sets $Y$ of $G$.
\end{theorem}

\begin{proof}
Consider a graph automatic graph $\Gamma(G,X)$. Let $Y$ be any finite generating set for $G$. 
Each $y\in Y$ can be written as a product $x_1^{k_1} \cdot \ldots \cdot x_n^{k_n}$ of elements of
$X$. We write this product as $w(y)$. By Lemma \ref{L:Product} the binary relation $\{(u,v) \in \Gamma(G,X)^2\mid v = uy\}$ is FA recognizable in $\Gamma(G,X)$
 This proves that $\Gamma(G,Y)$ is an automatic graph. Note that we did not need to change the automatic
representation of the vertex set of the  graph $\Gamma(G,X)$ in our proof.
\end{proof}


\section{Automatic vs graph automatic}

In this section we recall the definition of automaticity first introduced by Thurston, and compare it with
our definition of graph automaticity. We recall the definition of Thurston \cite{Thurston}

\begin{definition}
A group $G$ with a finite generator set $X$ is {\bf automatic} if
\begin{enumerate}
\item There exists a regular subset $L\subseteq X^\star$ such that
the natural mapping $v\rightarrow \bar{v}$, $v\in L$, from $L$ into $G$  is onto.
\item The set $ W_G=\{(u,v) \mid u, v\in L \ \& \  \bar{u}=\bar{v}$  in $G \}$ is regular.
\item For each $x\in X$, there exists an automaton $M_x$ that recognizes the relation:
$$
\{(u,v) \mid u,v \in L \  \mbox{and} \ \bar{u}=\overline{vx} \ \mbox{in $G$} \}.
$$
\end{enumerate}
The automaton $M$ for $L$, and automata $M_x$ are called an {\bf automatic structure} for the group $G$.
\end{definition}

As mentioned in the introduction, automatic groups are generator set independent, have decidable word problem (in quadratic time, and they are finitely presented. They are also  closed under finite free products, finite direct products,
and  finite extensions.

\smallskip

Examples of automatic groups include free abelian groups $Z^n$, hyperbolic groups, e.g. free groups, braid groups, and fundamental groups of many natural manifolds. Examples of non-automatic groups are $SL_n(\Z)$ and $H_3(\Z)$, the wreath product of $\Z_2$ with $\Z$, non-finitely presented groups, and Baumslag-Solitar groups.

\smallskip

There is one geometric property of automatic groups known as fellow traveller property \cite{Thurston}. It is roughly explained as follows. \ Let $M$ and automata $M_x$, $x\in X$, be an {\em automatic structure} for the group $G$ generated by $X$. For any two words $w_1, w_2$ recognized by $M$, if
$M_x$ accepts $(w_1, w_2)$ then the following property holds.
Start traveling at the same speed along the paths $w_1$ and $w_2$ in the Cayley graph; At any given time $t$ during the travel, the distance between $w_1(t)$ and $w_2(t)$ is uniformly bounded by a constant $C$.

\smallskip

We now recast the definition of graph automaticity through the following lemma whose proof immediately follows from the definitions:

\begin{lemma}
Let $\Gamma(G,X)$ be the Cayley graph of a group $G$ generated by a finite set $X$ viewed as the  structure above:
$$
(V, E_{x_1}, \ldots, E_{x_n}).
$$
Then $G$ is graph automatic if and only if the following conditions hold for some finite alphabet $\Sigma$:
\begin{itemize}
\item There is an FA recognizable language  $R \subseteq \Sigma^\ast$ and an onto mapping $\nu:R \to V$ for which the binary predicate $E(x,y) \subseteq R^2$ defined by $E(u,v) \leftrightarrow \nu(u) = \nu(v)$ is FA recognizable,
\item All the predicates $E_{x_1}, \ldots, E_{x_n}$ are FA recognizable with respect to the mapping $\nu$, that is for each $x\in X$ the
set $\nu^{-1}(E_x)=\{(u,v) \mid u, v\in R \ \mbox{and} \  \nu(u) x =\nu(v)\}$ is FA recognizable.
\end{itemize}
\end{lemma}


Thus, the definition of graph automaticity  differs from the  definition Thurston automaticity
in only one respect.  Namely,  it is not requird that $X = \Sigma$.
This immediately implies the following simple result showing that all automatic groups are graph automatic.

\begin{proposition}
Every automatic group is graph  automatic. \qed
\end{proposition}

However, the converse is not true. For instance, the  Heisenberg group  $\H_3(\Z)$ is graph automatic (Example \ref{Ex:Heisenberg}), but not automatic (see \cite{Thurston}). Later we will give more examples of such groups.

\section{The word and conjugacy problems}

Recall that the complexity of the word problem in each automatic group is bounded by a quadratic polynomial. The theorem below shows that graph automatic groups enjoy the same property. They behave just like automatic groups in terms of complexity of  the word problem.

\begin{theorem}
The word problem in graph automatic groups is decidable in quadratic time.
\end{theorem}

\begin{proof}
Let $G$ be a group for which the Cayley graph $\Gamma(G,X)$ is automatic.
We prove the following   result which is interesting in its own as it can be applied in the general setting:

\begin{lemma}
Let $f: D^n\rightarrow D$ be a function whose graph is FA recognizable. There exists a linear time algorithm that given $x_1, \ldots, x_n\in D$ computes the value $f(x_1, \ldots, x_n)$.
\end{lemma}

To prove the lemma, lets us denote by $\M$ a finite automaton recognizing the graph of $f$. Consider the set $X$ of all paths (runs) labeled by words of the form $\c(x_1,\ldots, x_n, y)$, where $|y|\leq max \{|x_i| \mid 1\leq i\leq n\}$ starting from the initial state $q_0$. Let $S$ be the set of all states obtained by selecting the last states in the paths from $X$. The set $S$ can be computed in time $C \cdot  max\{|x_1|,\ldots, |x_n|, |y)|\}$, where $C$ is a constant.  There are two cases for $S$:

\smallskip

{\em Case 1}: The set $S$ contains an accepting state $s$. Hence there exists a path from the initial state to $s$ such that the label of the path is of the form $\c(x_1,\ldots, x_n, y')$ with $|y'|\leq max \{|x_i| \mid 1\leq i\leq n\}$. One can find such a path in linear time on size of the input $\c(x_1,\ldots, x_n)$. The string $y'$ must be such that $f(x_1,\ldots, x_n)=y'$.

\smallskip

{\em Case 2}. The set $S$ does not contain an accepting state. There must exists a state $s\in S$ and  a path from $s$ to an accepting state $s'$ such that the path is labelled by $( \diamond, \ldots, \diamond, y'')$ such that $|y''| \leq C'$, where $C'$ is the number of states in $M$. Let $y'$ be a string of length $\con (x_1,\ldots, x_n)$ such that there is a path from $q$ to $s$ labeled by $\c(x_1,\ldots, x_n, y')$. Then $y=y'y''$ is the output of the function $f$ on input $(x_1,\ldots, x_n)$.
Note that finding $s'$ and $y'$ takes a linear time on the size of the input $\c(x_1,\ldots, x_n)$.
This proves the lemma.

\medskip

Now we prove the theorem.  Let $w=\sigma_1\ldots \sigma_n$ be a reduced word (that is a word over $X$ that does not contain sub-words of the form $xx^{-1}$ with $x\in X$). We would like to find a representation $u$ of this word in the group $G$ in the automatic representation of $\Gamma(G,X)$. For each $w_i=\sigma_1\ldots \sigma_i$ we can find a string $u_i$ representing $w_i$ such that $|u_i|\leq C_1\cdot i$ for some constant $C_1$. By the lemma above $u_i$ can be found in time $C_2\cdot i$ for some constant $C_2$. Hence, the word $u=u_n$ representing $w$ can be found in time \ $C_2(1+2+\ldots +n)=O(n^2)$.
This proves the theorem.
\end{proof}

Below we introduce a notion of a {\em Cayley graph biautomatic group}.  Let $G$ be a group generated by a finite set $X$.
 Let
$\Gamma(G,X)$ be the Cayley graph of $G$ relative to $X$. Consider  the {\em left   Cayley} graph $\Gamma^l (G,X)$. It is a labelled directed graph with the vertex set $G$ such that there is  a directed edge $(g,h)$ from $g$ to $h$ labelled by $x$  if  and only if $x g = h$.  The graph $\Gamma^l (G,X)$ can be viewed as an algebraic structure $\Gamma^l(G,X) = (G; E^l_{x_1}, \ldots, E^l_{x_n})$, where a binary predicate $E^l_{x_i}$ defines the edges with the label $x_i$ in $\Gamma^l(G,X)$.

\begin{definition}
A group $G$ generated by a finite set $X$ is {\bf Cayley graph biautomatic} if the graphs $\Gamma(G,X)$ and $\Gamma^l(G,X)$ are automatic relative to one and the same regular set representing $G$.
 Equivalently,  $G$ is Cayley graph biautomatic if and only if the structure $(G; E_{\sigma_1}, \ldots, E_{\sigma_n}, E^l_{\sigma_1}, \ldots, E^l_{\sigma_n})$ is automatic. Similar as above we often refer to these groups as {\bf graph biautomatic groups}. 
\end{definition}

Recall that biautomatic groups (in the sense of Thurston) are defined in the following way. Let $G$ be automatic group with respect to $X$. Let $L\subseteq X^{\star}$ be a part of automatic structure for $G$. We say that
$G$ is biautomatic if $L^{-1}$ is a part of automatic structure for $G$.

\begin{proposition}
Every Thurston biautomatic group is Cayley graph biautomatic.
\end{proposition}
\begin{proof}
Let $G$ be a Thurston biautomatic group with a finite generating set $X$. Suppose $R \subseteq X^\ast$  is a regular set such that $G$ is automatic relative to $R$ and $R^{-1}$. It follows that the Cayley graph $\Gamma(G,X)$ is automatic, so all the binary relations $E_{x_i}$ are FA recognizable. We need to show that the relations $E^l_{\sigma_i}$ are also FA recognizable. Since $G$ is biautomatic the set of pairs $(u,v) \in R^2$ such that $u^{-1}x^{-1} = v^{-1}$  for a given $x \in X$ is FA recognizable, say by an automaton $M_{x^{-1}}$. Observe that  $u^{-1}x^{-1} = v^{-1}$ if and only if $x u = v$.  Rebuild the automaton $M_{x^{-1}}$ into an automaton $M^l_x$ by  interchanging the sets of initial and final states in $M_x$, then  reversing each edge in $M_x$ and changing each label $x$ into $x^{-1}$. Clearly, $M_{x^{-1}}$ accepts a path with label $(u^{-1},v^{-1}) \in R^2$ if and only if $M^l_x$ accepts a path labelled $(u,v)$ (in which case  $v = x u$). Hence $M^l_x$ recognizes $E^l_{x}$. This proves the proposition.
\end{proof}

\begin{theorem}
The Conjugacy Problem in every graph biautomatic group $G$ is decidable.
\end{theorem}
\begin{proof}
Let $G$ be a graph biautomatic group generated by a finite set $X$.
Let $\Gamma(G,X)$ be a graph biautomatic representation of $G$,
with the regular set $R$ representing the domain.
Cayley graphs $\Gamma(G,X)$ and $\Gamma^l(G,X)$ are automatic.
Fix two words $p$ and $q$ in $X^\ast$. By Lemma \ref{L:Product}, applied to the automatic graphs $\Gamma(G,X)$ and $\Gamma^l(G,X)$, one has that the sets of pairs $\{(u,up) \mid u \in R\}$ and $\{(u,qu) \mid u \in R\}$ are FA recognizable. Hence, the set
$$S_{p,q} = \{ u \in R \mid up = qu \ in \ G\}$$
 is FA recognizable. Indeed, the formula
 $$
 \Phi(u)= \exists z ((up = z) \wedge (qu = z))
 $$
 defines the set $S_{p,q}$ in the automatic structure $(G; E_{\sigma_1}, \ldots, E_{\sigma_n}, E^l_{\sigma_1}, \ldots, E^l_{\sigma_n})$.  It follows, that  $p$ and $q$ are conjugate in $G$ if and only if $S_{p,q} \neq \emptyset$, which is decidable.
\end{proof}

Just as for automatic groups we do not, however, know if the Conjgacy Problem for graph automatic groups is decidable.

\section{Universal Cayley graphs}
\label{se:universal}

In this section we prove that the Cayley graph of a free group with two natural extra predicates is universal. Recall that an automatic structure $\A$ is {\bf universal} if every other automatic structure $\B$ can be interpreted in $\A$
as defined in Definition \ref{Dfn:Interp}.
\smallskip

Let $F$ be a free group with basis $A = \{a_1, \ldots,a_n\}$. We represent $F$ by the set $F(A)$ of all reduced words in $A \cup A^{-1}$. Recall that a word is reduced if it contains no subword of the form $aa^{-1}$, $a\in A$. On the set $F(A)$ define the following two predicates $\preceq$ and $el$:
$$
x\preceq y \leftrightarrow x \ \mbox{is  a  prefix of $y$, and}
$$
$$
el(x,y) \leftrightarrow |x| = |y|.
$$
Denote by $\Gamma_{free}(A)$ the Cayley graph $\Gamma(F,A)$ with two extra predicates $\preceq$ and $el$, i.e.,
$$
\Gamma_{free}(A) = (F(A); E_{a_1}, \ldots, E_{a_n}, \preceq, el).
$$
Now we prove the following theorem:

\begin{theorem}
The automatic structure  $\Gamma_{free}(A)$ is universal.
\end{theorem}
\begin{proof}
It is easy to see that $\Gamma_{free}$ is an automatic structure.
Indeed, the set  $F(A)$ andall the predicates defined are clearly FA recognizable.

Consider the structure
$$
{\mathcal M} = (A^\ast; R_a(x,y), x\preceq y, el(x,y))_{ a \in \Sigma},
$$
defined in Example  \ref{ex:free-monoid}.  By Theorem \ref{th:universal} $\M$ is a universal automatic structure. Since interpretability is a transitive relation,   it suffices to interpret the structure $\M$ in the expanded free group $\Gamma_{free}(A)$ by first-order formulas.
Notice, that the set $A^\ast$ is a subset of $F(A)$, consisting of all words without "negative" letters $a^{-1}$ when $a \in A$. Furthermore, all the predicates in $\M$ are restrictions of the corresponding predicates from $\Gamma_{free}(A)$ onto $A^\ast$. Hence, it suffices to show that the subset $A^\ast$ is definable in $\Gamma_{free}(A)$. Observe, first, that the formula
$$
\Phi_{<}(u,v)  = \exists z (z \preceq v \wedge z \neq v \wedge |z| = |u|)
$$
defines the binary relation $|u| < |v|$ in $\Gamma_{free}(A)$. Now it is easy to see that the formula
$$
\Phi(w) = \forall u \forall v (u \preceq w \wedge (\bigvee_{a \in A} R_a(u,v)) \to |u| < |v|)
$$
defines $A^\ast$ in $\Gamma_{free}(A)$. This proves the theorem.
\end{proof}

\section{Cayley graph automatic groups: constructions}

Our goal is to show that  graph automaticity is preserved under several natural group-theoretic constructions.

\subsection{Finite extensions}

Let $G$ be a group and $H$ be a normal subgroup $G$. We say that $G$ is a {\em finite extension of $H$} if the quotient group $G/H$ is finite. It turns out that graph automaticity preserves finite extensions:

\begin{theorem}\label{Thm:FE}
Finite extensions of graph automatic groups are again graph automatic.
\end{theorem}

\begin{proof}

Let $H$ be a graph automatic group. Let $H$ be a normal subgroup of a group $G$ such that $G/H$ is finite.  Let
$$
G/H=\{Hk_0, \ldots, Hk_{r-1}\}
$$
be all right co-sets of $G$ with respect to $H$.
There exists a finite function $g$ such that for all $i,s \leq r-1$, we have an equality:
\begin{equation}\label{eq:cosets}
 Hk_i \cdot H k_s=H k_{g(i,s)}.
\end{equation}
Let $h_0, \ldots, h_{n-1}$ be a finite number of generators of $H$ that also include the identity of the group. The equality  (\ref{eq:cosets}) above implies that there are sequences
$g_1(i,s)$, $\ldots$, $g_x(i,s)$  and $u_1(i,s)$, $\ldots$, $u_x(i,s)$ of integers  such that
we have
$$
k_ik_s=h_{u_1(i,s)}^{g_1(i,s)}, \ldots h_{u_x(i,s)}^{g_x(i,s)}k_{g(i,s)},
$$
where $i,s\leq r-1$ and all $u_1(i,s)$, $\ldots$, $u_x(i,s)$ are non-negative integers all less than or equal to $n-1$. Similarly, there are sequences  $f_1(i,j)$, $\ldots$, $f_m(i,j) $ and $v_1(i,j)$,$\ldots$, $v_m(i,j)$ of integers  such that for all $i\leq r-1$ and $j\leq n-1$ we have the following equalities:
$$
k_i h_j  = h_{v_1(i,j)}^{f_1(i,j)} \cdot \ldots \cdot h_{v_m(i,j)}^{f_m(i,j)}k_i.
$$
This implies that for all all $s,i\leq r-1$, $j\leq n-1$, and $h\in H$ we have the following equalities:
$$
h k_i \ h_j k_s  =h  \ h_{v_1(i,j)}^{f_1(i,j)} \cdot \ldots \cdot h_{v_m(i,j)}^{f_m(i,j)}k_i k_s=h \   h_{v_1(i,j)}^{f_1(i,j)} \cdot \ldots \cdot h_{v_m(i,j)}^{f_m(i,j)} \ h_{u_1(i,s)}^{g_1(i,s)} \ldots h_{u_x(i,s)}^{g_x(i,s)} \ k_{g(i,s)}.
$$
Let $\bar{h}$ be the word representing the element $h\in H$ under a graph automatic presentation of $H$. We represent elements $hk$ of the group $G$ as words $\bar{h}k$. Here we need to assume that the alphabet of the presentation for $H$ does not contain symbols $k_0$, $\ldots$, $k_{r-1}$. The equalities above tell us that there are finite automata $M_{i,j}$ that for every
$k_i, h_j$ accept all pairs of words of the form $(\bar{h}k, w)$ such that the equality $w=h k k_ih_j $ is true in the group $G$. Note that to build the automata $M_{i,j}$ one needs to use: the original automata that represent the group  $H$,  the sequences $g_1(i,s)$, $\ldots$, $g_x(i,s)$ and  $u_1(i,s)$, $\ldots$, $u_x(i,s)$,  the sequences
 $f_1(i,j)$, $\ldots$, $f_m(i,j)$ and $v_1(i,j),\ldots, v_m(i,j)$,   the function $g$, and the  automata  representing the multiplication by elements $h_v^{f(i,j)}$ and $h_u^{g(i,j)}$ in  the group $H$. This shows that the group $G$ is graph automatic.
The theorem is proved.  
\end{proof}

A simple corollary of the proof is the following:

\begin{corollary}\label{Cor:FE}
Finite extensions of graph biautomatic groups are again graph biautomatic.
\end{corollary}

\subsection{Semidirect products} \label{se:semidirect}

Let $A$ and $B$ be finitely generated groups and $\tau:B \to Aut(A)$ an injective homomorphism.
As usual the {\em semidirect product} of $A$ and $B$ relative to $\tau$, denoted $A \rtimes_\tau B$, is a  group $G$ generated by $A$ and $B$ such that $A$ is normal in $G$. In the semidirect product, every element $g \in G$ is uniquely presented as a product $g = ba$, where $a \in A, b \in B$. The
multiplication in $G$ is given by
$(ba)(b_1a_1) = bb_1a^{b_1}a_1$, where $a^{b_1} = \tau(b_1)(a)$.

\smallskip

Recall that an automorphism $\alpha \in Aut(A)$ is automatic if its graph is an FA recognizable language.

\begin{theorem} \label{Thm:Semidirect}
Let $A$ and $B$ be graph automatic groups with finite sets of generators $X$ and $Y$, and $\tau:B \to Aut(A)$ an injective homomorphism. Assume that the automorphism $\tau(y)$ is automatic  for every $y \in Y$. Then the semidirect product $G = A \rtimes_\tau B$ is graph automatic.
\end{theorem}
\begin{proof}
Let $R$ and $S$ be regular sets that give graph automatic representations of $A$ and $B$. The generators of the semidirect product are of the form $(e_B, x)$ and $(y, e_A)$, where $x\in X$, $y\in Y$,
and $e_A$, $e_B$ are units of $A$ and $B$ respectively. For any $x \in X$ the relation $(sr)(e_B,x) = s_1r_1$ is obviously FA recognizable. Similarly, for each $y\in Y$, we have  $(sr)(y, e_A) = syr^y$ and this relation is also FA recognizable, since the graph $\{(r,r^y) \mid r\in R\}$ is FA recognizale. This is because
 $\tau(y)$ is automatic. This proves the theorem.
\end{proof}

An immediate corollary of the theorem is this:

\begin{corollary}
Direct product of two graph automatic groups is graph automatic. \qed
\end{corollary}

Consider the group $G = (\Z \times \Z) \rtimes_A \Z$, where $A \in SL(2,\Z)$. Here we mean that the action of a generator, say $t$, of $\Z$  on $\Z \times \Z$ is given by the matrix $A$. Such groups play an important part as lattices in the Lie group $Sol = (\mathbb{R} \times \mathbb{R}) \rtimes \mathbb{R}$, where $t$ acts on $\mathbb{R} \times \mathbb{R}$ by a diagonal matrix $diag(e^t,e^{-t})$. These groups are also interesting in our context because of the following observation. If $A$ is conjugate in $GL(2,\mathbb{R})$ to a matrix $diag(\lambda, \lambda^{-1})$ for some $\lambda >1$, then $G$ has exponential Dehn function, hence $G$ is not an automatic group \cite{Thurston}. The theorem above can be applied to prove the
next result.

\begin{proposition} \label{pr:sol}
The group $G = (\Z \times \Z) \rtimes_A \Z$ is graph automatic for every $A \in SL(2,\Z)$.
\end{proposition}
\begin{proof}
We first note that every matrix $A \in SL(2,\Z)$ gives rise to an  FA recognizable automorphism of $\Z \times \Z$. Since the underlying groups $\Z$, $\Z\times \Z$ are graph automatic, by the theorem above the group $G$ is graph automatic.

\smallskip

Alternatively, graph automaticity of $G$ can also be shown via Theorem \ref{FDT}. Indeed,
the Cayley graph $\Gamma$ of $G$ is first-order interpretable in $(\Z; +)$, which is automatic.
To see this, represent elements of $G$  as triples $(x,y,t) \in \Z^3$. This set  is FA recognizable.
Now observe  that multiplication in $G$ is given by
$$
(x_1,y_1,t_1)(x_2,y_2,t_2) = ((x_1,y_1) +A(x_2,y_2)^T, t_1+t_2).
$$
Therefore  multiplication of $(x_1,y_1,t_1)$ by a fixed generator of $G$ is definable in $(\Z;+)$ as claimed.
\end{proof}
\subsection{Wreath products}

For the next theorem we define the {\em restricted wreath product} of a group $A$ by a group  $B$. Let $A_b$ be an isomorphic copy of $A$ for each $b\in B$. Consider the direct sum of groups $A_b$ denoted by $K$. Thus,
$$
K=\bigoplus_{b\in B} A_b,
$$
where elements of $K$ are functions $f: B\rightarrow A$ such that $f(b)=1_A$ for almost all $b\in B$.
We write elements of $K$ as $(a_b)$.
Each element $c \in B$ induces an automorphism $\alpha_{c}$ of $K$ as follows:
$$
\alpha_{c}(a_b)=(a_{bc}).
$$
The wreath product of $A$ by  $B$ consists of all pairs of the form $(b,k)$, where $b\in B$ and $k\in K$, with multiplication defined by:
$$
(b, k) \cdot (b_1, k_1)=(b b_1, \alpha_{b_1} (k)k_1).
$$
Thus, the wreath product of $A$ by $B$ is simply the semidirect product of
$K=\bigoplus_{b\in B} A_b$ and $B$ relative to the mapping $B\rightarrow Aut(K)$ given by  $c\rightarrow \alpha_c$.

\begin{theorem} \label{th:wreath}
For every finite group $G$ the wreath product of $G$ by $Z$ is graph automatic.
\end{theorem}

\begin{proof}
This follows directly from Theorem \ref{Thm:Semidirect}. However, we give an explicit automatic presentation
of the wreath product. The elements of the wreath product are of the form
$$
(i, (\ldots, g_{-n}, g_{-n+1}, \ldots, g_{-1}, g_{0}, g_{1}, \ldots, g_{m-1}, g_{m}, \ldots)),
$$
where  $g_j\in G$ and $i \in Z$. We refer to $g_0$ as the element of $G$ at position $0$. We can assume that $g_{k}$ is the identity $1_G$ of the group $G$ for all $k<-n$ or $k>m$, and
$g_{-n}\neq 1_G$  and $g_m\neq 1_G$. We can represent the element above as the following string
$$
\con (i, g_{-n} \ldots g_{-1} (g_0, \star) g_1  \ldots g_m),
$$
where $i$ is written in binary. The alphabet of these strings is clearly finite since $G$ is a finite group. The symbol $\star$ in this string represents  elements of $G$ at position $0$.  The generators of the wreath product are elements $(0,g)$ and $(1,g)$ represented by the strings
$\oplus(0,(g,\star))$ and $\oplus(1, (1_G,\star))$, where $g\in G$. Multiplication by these generators works as follows:
$$
\con (i, g_{-n} \ldots (g_0, \star)  \ldots g_m) \cdot \oplus (0,(g,\star))=\con (i, g_{-n} \ldots (g_0\cdot g, \star)  \ldots g_m)
$$
and
$$
\con (i, g_{-n} \ldots (g_0, \star)   \ldots g_m) \cdot \oplus (1,(1_G,\star))=\con (i+1, g_{-n+1}' \ldots (g_{0}', \star)   \ldots g_{m+1}'),
$$
where $g_{j+1}'=g_j$ for $j\in \{-n, \ldots, m\}$. These operations can clearly be performed by finite automata. The theorem is proved.
\end{proof}

The theorem above can be applied to construct many examples of graph automatic groups that are not finitely presented.
Hence, these give us another class of graph automatic but not automatic groups.

\begin{corollary}
There exist graph automatic but not finitely presented (and hence not automatic) groups.
\end{corollary}

\begin{proof}
The restricted  wreath product of a non-trivial finite group $G$ by $\Z$,  
by Theorem \ref{th:wreath},  is graph automatic. Now we use the following theorem by Baumslag \cite{Baumslag60}. For finitely presented groups  $A$ and $B$, the restricted wreath product of  $A$ by $B$ is finitely presented if and only if either $A$ is trivial or 
$B$ is finite. Hence, for nontrivial finite group $G$, the restricted wreath product of $G$ by $Z$ is not finitely presented but graph automatic.
\end{proof}

\subsection{Free products}

In this section we prove that  graph automaticity is preserved with
respect to free products.  The result follows from
representation of elements of the free product by their normal forms. 

\begin{theorem} \label{Thm:Free-product}
If $A$ and $B$ are graph automatic groups then their free product $A\star B$ is again graph  automatic.
\end{theorem}

\begin{proof}
Since $A$ and $B$ are graph automatic we can assume that the elements of $A$ and $B$ are strings over disjoint alphabets $\Sigma_1$ and $\Sigma_2$. Therefore,
$A\cap B=\{\lambda\}$. A {\em normal form} is a sequence of the type
$$
g=g_1 \Box g_2 \Box \ldots \Box g_n,
$$
where $g_i \in A\cup B$, $g_i \neq \lambda$, $\Box \not \in \Sigma_1 \cup \Sigma_2$, and the adjacent elements $g_i$ and $g_{i+1}$ are not from the same group $A$ or $B$, where $i\geq 0$.
The set $N$ of all normal forms  is FA recognizable. Every element in $A \ast B$ is uniquely represented by some normal form $g$ in $N$  and every normal form $g \in N$  gives rise to a unique element in $A \ast B$.   If  $a_1, \ldots, a_n$ generate $A$
and $b_1, \ldots, b_m$ generate  $B$  then these elements together  generate the whole group $A\star B$.
The multiplication by each of these generators can be performed by finite automata using the automata given for the underlying groups $A$ and $B$. For instance, the automaton $M$ that multiplies $g \in N$ by a generator $a\in A$ can be described as follows. Given $g, g' \in N$ the automaton $M$ reads $\con (g, g')$. The aim  is to detect if $ga=g'$ in $G$. Assume that
$$
g'=g_1' \Box g_2' \Box \ldots \Box g_n'.
$$
Notice that if $g_n\in B$ then $g'$ must be of the form:
$$
g'=g_1 \Box g_2 \Box \ldots \Box g_n \Box a.
$$
And if $g_n\in A$ then $g_n'$ must be of the form:
$$
g'=g_1 \Box g_2 \Box \ldots \Box g_{n-1} \Box g_n' ,
$$
where $g_na=g_n'$ in the group $A$. The last equality can be detected by a finite automaton using the automaton that recognizes the multiplication by $a$ in the group $A$.
\end{proof}

\subsection{Amalgamated products}

Let $A$ and $B$ be groups. Let $\phi$ be an isomorphism from a subgroup $H_A$ of $A$ into the subgroup $H_B$ of $B$. By $H$ we denote the isomorphism type of the group $H_A$. 
The {\em amalgamated product} of $A$ and $B$  by $H$, denoted by $A\star_{H} B$,
is the factor group of $A\star B$ by the normal closure of the set $\{\phi(h) h^{-1} \mid h\in H_A\}$.
The amalgamated product $A\star_H B$ is viewed as the result of identifying $H_A$ and $H_B$ in  the free product $A\star B$. 
Below we show simple conditions guaranteeing graph automaticity of amalgamated products.

\begin{theorem}
Let $A, B$ be graph automatic groups and $H$  a subgroup of $A$ and $B$. Suppose that one of the following conditions holds:
\begin{enumerate}
\item The group $H$ is graph biautomatic and $A$ and $B$ are finite extensions of $H$, or
\item  $H$ is a finite subgroup of both $A$ and $B$.
\end{enumerate}
Then the amalgamated product $A\star_H B$ is graph automatic.
\end{theorem}

\begin{proof} 
Assume that $H$ is a graph biautomatic group, and $A$ and $B$ are finite extensions of $H$. By Corollary \ref{Cor:FE}, both
$A$ and $B$ can be assumed to be graph biautomatic groups. Moreover, we can assume that the set of elements of the 
subgroup $H$ is a regular language.  

\smallskip

As in the proof of Theorem \ref{Thm:Free-product}, we consider normal forms. Note
that elements of $A$ and $B$ are strings (under given automatic presentations) since $A$ and $B$ are graph automatic.
We also assume that the alphabets $\Sigma_1$ and $\Sigma_2$ of graph automatic representations
of $A$ and $B$ are disjoint.

\smallskip

We choose the set of representatives $R_A$ and $R_B$ of the cosets of $H$ in $A$ and in $B$, respectively. 
These two sets are finite by the assumption. Let $a_1, \ldots, a_m$ and $b_1, \ldots, b_s$ be the strings from $R_A$ and $R_B$, 
respectively. Note that for every $g\in \{a_1, \ldots, a_m, b_1, \ldots, b_m\}$, by Lemma \ref{L:Product} and graph biautomaticity of $A$ and $B$,  the sets
$$
\{(u, v) \mid u,v \in A \ \& \ ug=v\}  \ \ \mbox{and} \ \ \{(u, v) \mid u,v \in A \ \& \ gu=v\}
$$
are FA recognizable sets.
Let $\Box$ be a symbol not in
$\Sigma_1\cup \Sigma_2$. Define an {\bf $A$-normal form} as a sequence
$$
g_1 \Box g_2 \Box \ldots \Box g_n \Box z,
$$
such that
\begin{enumerate}
\item each $g_i$ belongs to either $R_A$ or $R_B$, $g_i\neq \lambda$,
\item two consecutive $g_i$ and $g_{i+1}$ belong to distinct set of represenatives, and
\item $z\in H_A$.
\end{enumerate}

It is not hard to see that the set of all $A$-normal forms is a regular language that we denote by $N$.
This set determines elements of the amalgamated group $A \star_{H} B$. Moreover, each element in the amalgamated
group has a unique representation in $A$-normal form \cite{Oleg}.

\smallskip

Let $X$ and $Y$ be finite set of generators for $A$ and $B$, respectively.
Our goal is now to show that the multiplication of elements in $N$ by each of the generator elements from
$X\cup Y$ is FA recognizable. We first consider the case when the generator is in $X$. Take $v\in N$ of the form
$g_1 \Box g_2 \Box \ldots \Box g_n \Box z$, and a generator $x\in X$. 

\smallskip

Assume that $g_n\in R_B$.  Using the fact that $A$ is automatic, we can compute (the string representing) the element  $z\cdot x$.
Let $w$ be the element $z\cdot x$. We now find an element $a \in R_A$ such that  $w \in aH$. Hence, we can write
the element $w$ as $a\cdot h$ for some $h\in H$. Namely, $w=a(a^{-1}w)$. Thus, we have
the equality:
$$
vx=g_1 \Box g_2 \Box \ldots \Box g_n \Box a \Box a^{-1}w. 
$$
From the above, since the underlying groups are biautomatic, we see that this is a FA recognizable event, that is the set
$$
\{(v, v') \mid \mbox{$v$ is of the form $g_1 \Box g_2 \Box \ldots \Box g_n \Box z$, $v'\in N$, $g_n\in B$, $x\in X$, $v'=vx$} \}
$$
is FA recognizable. 

\smallskip

Assume now that $g_n \in R_A$. Since $A$ is biautomatic the set 
$$
\{(g_n,w) \mid  w= g_n\cdot z \cdot x, \  g_n\in R_A, \ z\in H, \ x\in X\}
$$
is FA recognizable. Now given $w=g_n\cdot z \cdot x$, we can represent it as the product $ah$ for some $a\in R_A$ and $h\in H$.
This is again a FA recognizable event. We conclude that the set multiplication by $x\in X$ of elements in $N$ can be recognized by finite automata. The case when we multiply elements of $N$ by the generators $y\in Y$ is treated similarly. 

\smallskip

Now we prove the second part of the theorem. Since $H$ is finite the set of all left co-sets with respect to $H$ in both $A$ and $B$ 
is uniformly FA recognizable. In other words, the sets
$$
\{(a_1, a_2) \mid a_1^{-1}a_2\in H, \ a_1, a_2\in A\} \ \mbox{and} \ \{(b_1, b_2) \mid b_1^{-1}b_2\in H, \ b_1, b_2\in B\} 
$$
are FA recognizable languages. Therefore we can select regular sets $R_A$ and $R_B$ of left-cost representatives of $A$ and $B$, respectively. As above, one considers the set $N$ of normal forms. For each $z\in H$ and $g\in X\cup Y$, for each of the groups 
$A$ and $B$ there exists a finite automaton that recognizes the language $\{(u,v)\mid uzg=v\}$. Therefore, since $H$ is finite, we have that for each $g\in X \cup Y$ the set
$$
\{(v,w)\mid v,w \in N, \ vg=w\}
$$
is FA recognizable. This proves the second part of the theorem. 
\end{proof}

As an application we give the following result.

\begin{corollary}
The groups $SL_2(\Z)$ and $GL_2(Z)$ are graph automatic.
\end{corollary}

\begin{proof}

The group $SL_2(\Z)$ is isomorphic to  $\Z_4\star_{\Z_2} \Z_6$. Similarly, the group $GL_2(\Z)$  is isomorphic to $D_4 \star_{D_2} D_6$, where $D_n$ is a Dihedral group (see for instance \cite{Oleg}). By the theorem above the groups $SL_2(\Z)$ and $GL_2(Z)$ both are graph automatic.
\end{proof}

Of course, the groups $SL_2(\Z)$ and $GL_2(Z)$  are already known to be automatic (see \cite{Thurston}). 

\section{Subgroups}

In this section we describe a simple technique that is analogues, in some respect, to the technique of quasi-convex subgroups in Thurston automatic groups. \ Let $G$ be a graph automatic group. Assume that $R$ is a regular language representing the group $G$ via a bijection $\nu:R \to G$.
A finitely generated subgroup $H \leq G$ is called {\em regular} if the pre-image $\nu^{-1}$ is a regular
subset of $R$. A similar definition describe regular subgroups of graph automatic monoids.

\begin{proposition} \label{pr:regular-subgroup}
Let $G$ be a graph automatic group  or monoid and  $H$ a regular finitely generated subgroup of $G$. Then $H$ is graph automatic.
\end{proposition}
\begin{proof}
Suppose that $R$ is a regular language representing the group $G$ via a bijection $\nu:R \to G$. Let $X_H$
be a finite generating set of $H$. Let $X_G$ be a finite generating set of $G$ containing $X_H$.  Since $G$ is graph automatic it is automatic relative to $X_H$. Hence the binary predicates $E_{x}$ are FA presentable in $R$.  Their restrictions to $\nu^{-1}(H)$ are  FA presentable as well.  Thus, $H$ is graph automatic.
\end{proof}

The following result, implied by the proposition above,  turns out to be useful in applications.

\begin{corollary} \label{pr:regular-subgroup-matrix}
Let $M_n(\Z)$ be the multiplicative monoid of all $n \times n$ integer matrices. If $H$ is a regular finitely generated subgroup of $M_n(\Z)$ then $H$ is graph automatic.
\end{corollary}

\begin{proof}
It suffices to note that multiplication operation of matrices in $M_n(\Z)$  by a fixed $n \times n$-matrix is
an automatic operation.
\end{proof}



\section{Nilpotent Cayley graph automatic groups}


In this section we show that there are many interesting  finitely generated nilpotent groups which are graph automatic.  For this we need to introduce a particular technique of polycyclic presentations that initially comes from Malcev's work on nilpotent groups \cite{Mal1}. For  the detailed  exposition see the books \cite{KargMerz,HoltBook,BaumslagNil}.

\smallskip

Let $G$ be a group, $a =(a_1, \ldots,a_n)$  an $n$-tuple of elements in $G$, and $\alpha = (\alpha_1, \ldots, \alpha_n)$ an $n$-tuple of integers. By $a^{\alpha}$ we denote the following product
$$
a^{\alpha} = a_1^{\alpha_1}a_2^{\alpha_2} \ldots a_n^{\alpha_n}.
$$
Concatenation of two tuples $a$ and $b$ is denoted $ab$ and  a 1-tuple $(x)$ is usually denoted by  $x$.

\smallskip

Recall  that every finitely generated abelian group $A$ is a direct sum of cyclic groups:
$$
A = \langle a_1\rangle \times \ldots \langle a_s\rangle \times \langle b_1\rangle \times \ldots \langle b_t\rangle
$$
where $\langle a_i\rangle$ is an  infinite cyclic, and $\langle b_i\rangle$ is a finite cyclic of order $\omega(b_i)$. Every element $g \in A$ can uniquely  be represented in the form
\begin{equation} \label{eq:base}
g = a_1^{\alpha_1} \ldots a_s^{\alpha_s}b_1^{\beta_1} \ldots b_t^{\beta_t}
\end{equation}
where $\alpha_i \in \mathbb{Z}$ and $\beta_j \in \{0,1,\ldots, \omega(b_j)-1\}$. We call the tuple
$$
a = (a_1, \ldots, a_s,b_1,\ldots,b_t)
$$
a {\em base} of $A$ and the tuple $\sigma(g) = (\alpha_1, \ldots,\alpha_s,\beta_1, \ldots,\beta_t)$ the  coordinate of $g$ in the base $\bar a$. In this notation we write the equality (\ref{eq:base})  as folows $g = a^{\sigma(g)}$.


\smallskip

One can generalize the notion of  base to polycyclic groups.  Recall that a group $G$ is {\bf polycyclic} if there  is a sequence of elements $a_1, \ldots, a_n \in G$ that generates $G$ such that if $G_i$ denotes the subgroup $\langle a_i, \ldots, a_n \rangle$   then  $G_{i+1}$ is normal in $G_i$ for every $i$. In this case
\begin{equation}\label{eq:polycyclic-series}
G = G_1 \geq G_2 \geq \ldots \geq G_n \geq G_{n+1} = 1
\end{equation}
is termed a {\bf polycyclic series} of $G$. The sequence $a=\langle a_1, \ldots, a_n \rangle$ is called a base of $G$.

\smallskip

Let $a = (a_1, \ldots,a_n)$ be a base of a polycyclic group $G$.  Then the quotient $G_i/G_{i+1}$ is a cyclic group generated by the coset  $a_iG_{i+1}$. Denote by $\omega_i$ the order of the group $G_i/G_{i+1}$  which is the the order of the element $a_iG_{i+1}$ in $G_i/G_{i+1}$. Here
$\omega_i = \infty$ if the order is infinite. We refer to the tuple  $\omega(a) = (\omega_1, \ldots, \omega_n)$ as the {\bf order } of $a$. Now set  $Z_{\omega_i} = \mathbb{Z}$ if $\omega_i = \infty$ and $Z_{\omega_i} = \{0,1,2, \ldots, \omega_i -1\}$ otherwise.

\begin{lemma}
Let $a = (a_1, \ldots,a_n)$ be a base of a polycyclic group $G$ of order $\omega(a) = (\omega_1, \ldots, \omega_n)$ . Then for every $g \in G$ there is a unique decomposition of the following form:
\begin{equation}\label{eq:coordinates}
g = a_1^{\alpha_1} a_2^{\alpha_2} \ldots a_n^{\alpha_n} ,  \ \ \ \alpha_i \in Z_{\omega_i}.
\end{equation}
\end{lemma}
\begin{proof}
Let $g \in G$.  Since the quotient group $G_1/G_2$ is cyclic generated by $a_1G_2$ one has $gG_2 = a_1^{\alpha_1}G_2$ for some unique $\alpha_1 \in Z_{\omega_1}$.  The  element $g' = a_1^{-\alpha_1}g$ belongs to $G_2$. Notice that $(a_2, \ldots, a_n)$ is a base of $G_2$.  Hence by induction on the length of the base there is a unique decomposition of $g'$ of the type
$$
g' =  a_2^{\alpha_2} \ldots a_n^{\alpha_n} ,  \ \ \ \alpha_i \in Z_{\omega_i}.
$$
Now $g = a_1^{\alpha_1}g'$ and the result follows.
\end{proof}

In the notation above for an element $g \in G$ the tuple $\sigma(g) = (\alpha_1, \ldots, \alpha_n)$  from (\ref{eq:coordinates}) is called the tuple of {\bf coordinates } of $g$ in the base $a$. Sometimes we write the equality (\ref{eq:coordinates}) as $g = a^{\sigma(g)}$.

\smallskip

Finitely generated nilpotent groups are polycyclic so they have finite bases as above. Moreover, it is easy to see that an arbitrary  finitely generated group $G$ is nilpotent if and only if it has a finite base $(a_1, \ldots, a_n)$ such that the series (\ref{eq:polycyclic-series}) is  central, i.e., $[G_i,G] \leq G_{i+1}$ for every $i = 1, \ldots, n$ (here $G_{n+1} = 1$).

\smallskip

Now suppose $G$ is an arbitrary finitely generated nilpotent group of nilpotency class $m$.
The {\bf  lower central  series}  of $G$ is defined inductively by
$$
G_1 = G, G_2=[G_1, G], \ldots,  G_{i+1} = [G_i,G], \ldots
$$
By assumption,  we have $G_m \neq 1$ and $G_{m+1} = 1$. It follows that  all the quotients $G_i/G_{i+1}$ are finitely generated abelian groups. Let $d_i$ be a tuple of elements from $G_i$ such that its image in $G_i/G_{i+1}$ under the standard epimorphism is a base of the abelian group $G_i/G_{i+1}$.  Then the  tuple   $a = d_1d_2 \ldots d_m$ obtained by concatenation from the tuples $d_1, \ldots, d_m$, is a base of $G$. We refer to $a$ as a  {\bf lower central series} base of $G$.

\smallskip

Similarly, the {\bf  upper central series } of the group $G$ is the sequence:
$$
1=Z_0(G) \unlhd Z_1(G) \unlhd Z_2(G) \unlhd \ldots \unlhd Z_{i+1}(G) \unlhd \ldots,
$$
where $Z_{i+1}(G)$ is defined inductively as the set
$$
Z_{i+1}(G)=\{x\in G \mid \forall y \in G ([x,y]\in Z_i(G)\}.
$$
In particular $Z_1(G)$ is the center of $G$. Thus, the group $Z_{i+1}(G)$ is the full preimage of the center of the group $G/Z_i(G)$ under the canonical epimorphism $G \to G/Z_i(G)$.
If $G$ is  torsion-free then the quotients $Z_{i+1}(G)/Z_i(G)$ are free abelian groups of finite rank. In particular, one can choose a tuple $d_i$ of elements  from $Z_{i+1}(G)$ which form a standard basis of the free abelian group $Z_{i+1}(G)/Z_i(G)$, where $i = 1, \ldots,m$. The tuple
$a = d_m d_{m-1}  \ldots  d_1$ obtained
as concatenation of $d_m$, $\ldots$,  $d_1$ is called an {\bf upper central   base} of $G$. Notice that in this case  $\omega_i = \infty$ for each $i = 1, \ldots,n$. Such bases are  called {\em Malcev's bases} of $G$.

\smallskip

Now we give the following important definition that singles out  special type of polynomials needed
to perform the multiplication operation in finitely generated nilpotent groups of nilpotency class $2$.

\begin{definition}
We say that $p(x_1, \ldots, x_n, y_1, \ldots, y_n)$ is a {\bf special quadratic} polynomial in  variables $x_1, \ldots,x_n$ and $ y_1, \ldots, y_m$ if
$$p(x_1, \ldots, x_n, y_1, \ldots, y_n) = \Sigma_{i,j} \alpha_{ij}x_iy_j + \Sigma_i \beta_ix_i + \Sigma_j\gamma_jy_j,
$$
where $\alpha_{i,j}, \beta_i, \gamma_j$ are constants from $\mathbb{Z}$.
\end{definition}

In a tuple notation we write $p(x,y)$, where $x = (x_1, \ldots,x_n)$ and $y = (y_1, \ldots,y_n)$.
If $\alpha$ and $\beta$ are tuples of integers, then $f(\alpha,\beta)$ denotes the value of $f$
obtained by substituting  $x \to \alpha, y \to \beta$.
Similarly, for a tuple of polynomials $f(x,y) = (f_1(x,y), \ldots, f_k(x,y))$, we write
$f(\alpha,\beta)$ to denote $(f_1(\alpha,\beta), \ldots, f_k(\alpha,\beta))$. \ The following lemma indicates the use of special quadratic polynomials in calculating the group operation in a finitely generated group of nilpotency class at most $2$.

\begin{lemma} \label{th:coordinate-functions}
Let $G$ be a finitely generated 2-nilpotent group with a lower series base $a  = (a_1, \ldots,a_n)$.
There exist  a tuple of special polynomials $f(x,y) = (f_1(x,y), \ldots, f_n(x,y))$  with $x = (x_1, \ldots,x_n)$ and $y = (y_1, \ldots,y_n)$ such that for any tuples of integers $\alpha, \beta \in \mathbb{Z}^n$ one has
$$a^\alpha\cdot a^\beta = a^{f(\alpha,\beta)}.$$
\end{lemma}
\begin{proof}
Since $G$ is a 2-nilpotent group then $G > [G,G] > 1$ is the lower central series of $G$.  In particular, $[G,G]$ is a subgroup of the center $Z(G)$ of $G$.  Let   $\omega(a)$ be the order of the lower series base $a  = (a_1, \ldots,a_n)$. By definition of the base $a$, we have that $a$  is a concatenation of two tuples $d_1 = (a_1,\ldots,a_s)$ and $d_2 = (a_{s+1}, \ldots,a_m)$ such that $d_2$ is a base of the abelian group $[G,G]$.
For tuples $\alpha = (\alpha_1, \ldots, \alpha_n), \beta = (\beta_1, \ldots, \beta_n) \in \mathbb{Z}^n$ consider the following product
\begin{equation}
\label{eq:product1}
a_1^{\alpha_1}  a_2^{\alpha_2} \ldots  a_n^{\alpha_n}  \cdot a_1^{\beta_1}a_2^{\beta_2} \ldots a_n^{\beta_n}
\end{equation}

Since $G$ is 2-nilpotent all the commutators $[a_i^{\alpha_i},a_1^{\beta_1}]$ are in the center of $G$, so using equalities $a_i^{\alpha_i}a_1^{\beta_1} = a_1^{\beta_1}a_i^{\alpha_i}[a_i^{\alpha_i},a_1^{\beta_1}]$ one can rewrite the product (\ref{eq:product1}) in the following form:
  \begin{equation}\label{eq:product1a}
a_1^{\alpha_1+\beta_1}a_2^{\alpha_2} \ldots  a_n^{\alpha_n} a_2^{\beta_2} \ldots a_n^{\beta_n} \Pi_{i= 2}^n[a_i^{\alpha_i},a_1^{\beta_1}]
\end{equation}
By induction on the length of the base, we can assume that there are special quadratic polynomials, say $g_2({\bar x},{\bar y}), \ldots, g_n({\bar x},{\bar y})$, where ${\bar x} = (x_2, \ldots,x_n) ,{\bar y} = (y_2, \ldots,y_n)$, such that
$$
a_2^{\alpha_2} \ldots  a_n^{\alpha_n} a_2^{\beta_2} \ldots a_n^{\beta_n} = a_2^{g_2(\alpha,\beta)} \ldots a_n^{g_n(\alpha,\beta)}.
$$
Notice that for 2-nilpotent groups the equalities
$$[a_i^{\alpha_i},a_1^{\beta_1}] = [a_i,a_1]^{\alpha_i\beta_1}$$
hold  for every $\alpha_i, \beta_1$. So
$$
\Pi_{i= 2}^n[a_i^{\alpha_i},a_1^{\beta_1}] = \Pi_{i= 2}^n [a_i,a_1]^{\alpha_i\beta_1}
$$
Since $[a_i,a_1] \in [G,G]$ one has $[a_i,a_1] = a_{s+1}^{\delta_{s+1,i}} \ldots a_n^{\delta_{ni}}$ for some $\delta_{ji} \in \mathbb{Z}$. Therefore
$$
\Pi_{i= 2}^n [a_i^{\alpha_i},a_1^{\beta_1}] = \Pi_{i= 2}^n (a_{s+1}^{\delta_{s+1,i}} \ldots a_n^{\delta_{ni}})^{\alpha_i\beta_1} =
$$
$$
\Pi_{i= 2}^na_{s+1}^{\delta_{s+1,i}\alpha_i\beta_1} \ldots a_n^{\delta_{ni}\alpha_i\beta_1} = a_{s+1}^{\Sigma_{i= 2}^n\delta_{s+1,i}\alpha_i\beta_1} \ldots a_n^{\Sigma_{i= 2}\delta_{ni}\alpha_i\beta_1}
$$
Observe that $h_j(x_1, \ldots,x_n,y_1) = \Sigma_{i= 2}\delta_{ji}x_iy_1$ are special quadratic polynomials in $x$ and $y$, so
$$
\Pi_{i= 2}^n [a_i^{\alpha_i},a_1^{\beta_1}] = a_{s+1}^{h_{s+1}(\alpha,\beta)} \ldots a_n^{h_n(\alpha,\beta)}
$$

Combining the latter one with the equality (\ref{eq:product1a})  one gets that the initial product (\ref{eq:product1}) is equal to
$$
a_1^{\alpha_1+\beta_1}a_2^{g_2(\alpha,\beta)} \ldots a_s^{g_s(\alpha,\beta)} a_{s+1}^{g_{s+1}(\alpha,\beta) +h_{s+1}(\alpha,\beta)}    a_n^{g_n(\alpha,\beta)+h_n(\alpha,\beta)},
$$
which proves the lemma.
\end{proof}

\begin{theorem}
Every  finitely generated group $G$ of nilpotency class at most two is graph  automatic.
\end{theorem}

\begin{proof} We prove the theorem by cases.

\smallskip

{\em Case 1}: If $G$ is abelian then our Example \ref{Ex:Abelian} shows that $G$ is graph automatic.

\smallskip

{\em Case 2}: Assume that $G$ be a finitely generated torsion free 2-nilpotent  group.
Fix an arbitrary upper central Malcev's base $a$ of $G$.  We use notation from Lemma \ref{th:coordinate-functions} throughout the proof. Every element $g \in G$ can be uniquely represented by its tuple of coordinates $\sigma(g)$  relative to the base $a $.  The set of coordinates of elements of $G$
$$
\sigma(G) = \{\sigma(g) \mid g \in G\} = Z_{\omega_1} \times \ldots Z{\omega_n} = \mathbb{Z}^n
$$
is in bijective correspondence with $G$. This set is clearly definable by first-order formulas in the Presburger arithmetic ${\mathcal P} = \langle \mathbb{N}, +\rangle$.  By Theorem \ref{FDT}, the set is FA recognizable.

\smallskip

Now we prove that the Cayley graph $\Gamma$ of $G$ relative to the generating set $\{a_1, \ldots,a_n\}$ is interpretable in ${\mathcal P}$.   It suffices to show that  for  a given generator $a_i$ the set of pairs
$$\{(\sigma(g),\sigma(ga_i)) \mid g \in G\}$$
 is first order interpretable in ${\mathcal P}$ for each $a_i$.  To this end let $g \in G$ and  $\sigma(g) = (\alpha_1, \ldots, \alpha_n)$.  By Lemma \ref{th:coordinate-functions} there exist  a tuple of special polynomials $f(x,y)$ such that

\begin{equation}\label{eq:product3}
 ga_i = a^{f(\sigma(g),\varepsilon_i)}
\end{equation}
 where $\varepsilon_i = (0, \ldots,0,1,0,\ldots,0)$ (all components are equal to 0, except for the $i$s, which is equal to 1).
 Therefore,
 $$\{(\sigma(g),\sigma(ga_i)) \mid g \in G\} = \{(\sigma(g),f(\sigma(g),\varepsilon_i)) \mid g \in G\}
$$
Notice, that $f(\sigma(g),\varepsilon_i) = (f_1(\sigma(g),\varepsilon_i), \ldots, f_n(\sigma(g),\varepsilon_i)$
and every $f_j(\sigma(g),\varepsilon_i)$ is a fixed linear function in $\sigma(g)$ since $f$, by Lemma \ref{th:coordinate-functions}, is a special polynomial.  each linear polynomial is first order definable in $\mathcal P$. Therefore the set
$$(\alpha, f_j(\alpha,\varepsilon_i) ) \mid \alpha \in \mathbb{Z}^n\}
$$
is first order definable in $\mathcal P$ for every $j = 1, \ldots, n$. Hence the set
$$
\{(\alpha,f(\alpha,\varepsilon_i)) \mid \alpha \in \mathbb{Z}^n\}
$$
is also first order definable in $\mathcal P$ for every $i = 1, \ldots,n$.  All these are FA recognizable by Theorem \ref{FDT}. Thus $G$ is graph automatic.

\smallskip

{\em Case 3}.  Let $G$ be an arbitrary finitely generated 2-nilpotent group.   Then the set of all torsion elements in $G$ forms a finite subgroup $T(G)$ of $G$.  If $k$ is the order of $T(G)$ then the subgroup
$ G^k $ generated by $\{g^k \mid g \in G\}$ is a finitely generated torsion-free 2-nilpotent subgroup of $G$ such that  quotient $G/G^k$ is finite.  So $G$ is a finite extension of $G^k$. By Case 2 the group $G^k$ is graph automatic.  By Theorem \ref{Thm:FE}, the original group $G$ is also graph automatic. This proves the theorem.
 \end{proof}

There are finitely generated nilpotent graph automatic groups which are not 2-nilpotent. For instance, the group  $\H_n(Z)$, where $n>3$, as proved in Example \ref{Ex:n-Heidelberg} are graph automatic. The following provide other  examples of graph automatic nilpotent groups of class $>2$.

\begin{example}
The following groups are graph automatic:
\begin{itemize}
\item Let $UT(n,\mathbb{Z})$ be the group of upper triangular matrices over $\mathbb{Z}$ (with $1$ at the diagonal).

\item The group $UT^m(n,Z)$ that consists of all matrices from $UT(n,Z)$ such that the first $m-1$ diagonals above the main one have all entries equal to  $0$.
\end{itemize}
\end{example}

\begin{proof}
The proof follows from Proposition \ref{pr:regular-subgroup-matrix}.

\end{proof}

\section{Solvable graph automatic groups}

\subsection{Baumslag-Solitar groups}
The Baumslag-Solitar groups are finitely generated one-relator groups \cite{BaumSol62}. They play an important role in combinatorial and geometric group theory. These groups have two generators $a$ and $b$
and have parameters $n,m \in \N$. For each $n,m\in \N$ the presentation of the Baumslag-Solitar group $B(m,n)$ is given by the following relation:
$$
a^{-1}b^ma=b^n.
$$
It is well-known that the groups $B(m,n)$, for $m\neq n$, are not automatic, and  when $n=m$, the groups $B(n,m)$ are automatic \cite{Thurston}. It is also known that the Baumslag-Solitar groups are all asynchronously automatic \cite{Thurston}. In this section we prove that the Baumslag-Solitar groups $B(1,n)$ are graph automatic groups for all $n\in \N$.

\begin{theorem}\label{Thm: B(1,n)}
The Baumslag-Solitar groups $B(1,n)$ are graph automatic for $n\in \N$.
\end{theorem}

\begin{proof}
To simplify our exposition we consider the group $B(1,2)$. We prove this theorem through the action of this group on the real line $\R$. We represent the elements $a$ and $b$ of the group $B(1,2)$ as the linear functions $g_a:\R \rightarrow \R$ and
$g_b: \R \rightarrow \R$ given by $g_a(x)=2x$ and $g_b(x)=x+1$.  Let $G$ be the group generated by the linear functions $g_a$ and $g_b$. The group operation in $G$ is the composition of functions. Our goal is to show that the group $G$ is isomorphic to $B(1,2)$ via the isomorphism induced by the mapping $a\rightarrow g_a$ and $b\rightarrow g_b$.

\begin{claim}
The elements $g_a$ and $g_b$ satisfy the identity $g_a^{-1} g_b g_a=g_b^2$.
\end{claim}

Indeed, given a real number $x\in \R$, we have the following equalities:
 $$
 g_a^{-1} g_b g_a (x)=g_b g_a (\frac{1}{2} x)=g_a(\frac{1}{2}x+1)=x+2=g_b^2(x).
 $$



\medskip

For the next claim recall that $\Z[1/2]$ is the set of all dyadic numbers, that is numbers of the form
$i/2^j$, where $i,j \in \Z$.

\begin{claim}
Each $g\in G$ is a linear function of the form $ax+b$, where $a=2^n$ and $b\in \Z[1/2]$.
\end{claim}

The proof of the claim is  by induction on the length of words over $a,b$ representing elements of $G$. For $g_a$ and $g_b$ the claim is obvious.
Suppose $g\in G$ is of a desired  form $2^n x+m/2^k$. We need to show that the functions
$gg_a$, $g g_a^{-1}$, $gg_b$, and $gg_b^{-1}$ are also of the desired form. But this  can
be shown  through easy calculations. For instance,
\begin{center}{
$gg_a(x)=g_a(g(x))=g_a(2^n x+m/2^k)=2^{n+1}+m/2^{k-1}$, and \\
$gg_b(x)=g_b(g(x))=g_b(2^n x+m/2^k)=2^{n+1}+(m+2^k) /2^{k}$.
}
\end{center}

The next claim shows that every function of the form $2^nx+m/2^k$ can be generated through the base functions $g_a$ and $g_b$. This reverses the claim above.

\begin{claim}
Assume that $g$ is a function of the form $2^nx+m/2^k$. Then $g=g_a^n(g_a^{k} g_b^m g_a^{-k})$.
\end{claim}

Indeed, first note the following equality:
$$
g_a^{k} g_b^m g_a^{-k}(x)=g_b^m g_a^{-k} (2^k x)=g_a^{-k} (2^kx+m)=x+m/2^k.
$$
Now it is easy to see that
$$
g^n g_a^{k} g_b^m g_a^{-k}(x)=g_a^{k} g_b^m g_a^{-k}(2^nx)=2^nx+m/2^k=g
$$

These claims show that the groups $B(1,2)$ and $G$ are isomorphic. The isomorphism is induced by the mapping $a\rightarrow g_a$ and $b\rightarrow g_b$. So, we identify these two groups.

\smallskip

Now we give a representation of the Cayley graph for $B(1,2)$ with the generators $a$ and $b$. Consider a function $g\in G$ of the form $2^nx+m/2^k$, where $k\geq 0$. We can always assume that $m$ is odd if $k>0$. Thus, we can represent the element $g$ as the (convoluted) string $\con (n,m,k)$. We put the following conditions on these strings:
\begin{enumerate}
\item $n$ and $m$ are integers written in binary.
\item The integer $k$ is written in  unary.
\item If $k$ is the empty string (thus $k$ represents $0$), then $m\in \Z$. Otherwise, $m$ is odd.
\end{enumerate}
We denote this set of strings by $D$. It is clear that $D$ is finite automata recognizable set. It is also clear that the mapping $D\rightarrow G$ given by $(n,m,k)\rightarrow 2^nx+m/2^k$ is a bijection.

\smallskip

The multiplication by generators $g_a$ and $g_b$ of elements $g=2^nx+m/2^k$ of $G$ is now represented on $D$ as follows:
$$
(n,m,k) \rightarrow_a (n+1,m,k-1)  \ \mbox{and} \ (n,m, k) \rightarrow_b (n,m+2^k, k).
$$
It is clear that the multiplication by $a$ is recognized by finite automata. The multiplication by $b$ is also finite automata recognizable because $k$ is represented in unary.
\end{proof}

\subsection{Other metabelian groups}

We have shown in Section \ref{se:semidirect} that the groups $G = (\Z \times \Z)\rtimes_A \Z$ are graph automatic, where  $A \in SL(2,\Z)$. This can clearly be generalized to higher dimensions:

\begin{proposition}
A group $G = \Z^n \rtimes_A \Z$, where $A \in SL(n,\Z)$ is graph automatic.
\end{proposition}
\begin{proof}
The argument from Proposition \ref{pr:sol} can be applied here as well.
\end{proof}

\subsection{Non-metabelian solvable groups}

Let $T(n,\Z)$ be the group of triangular matrices of size $n \times n$  over the intgeres  $\Z$.
So, matrices in $T(n,\Z)$ have all zeros below the main diagonal.

\begin{proposition}
The group $T(n,\Z)$ is graph automatic.
\end{proposition}
\begin{proof}
This follows from Proposition \ref{pr:regular-subgroup-matrix}.
\end{proof}

Another interesting example comes from Theorem \ref{th:wreath}.

\begin{example}
Let $K$ be a solvable finite group.
Then by Theorem \ref{th:wreath} the wreath product   $K$ by the group  $\Z$ is Cayley graph
automatic. It is clear that $G$ is solvable of the solvability class at least the class of $K$.
\end{example}

\section{Proving non-automaticity}

In this section we discuss the issue of building non graph automatic groups.
Graph automatic groups, as we have proved, have decidable word problem.
Therefore, all finitely generated groups with undecidable word problem are obviously not graph
automatic. By Theorem \ref{FDT}, it is clear that if a structure $\A$ has undecidable first order theory then
$\A$ is not automatic. This suggests the following idea to construct a non graph automatic group  with decidable word problem. Search for a group $G$ with solvable word problem such that the first order theory of one of its Cayley graphs is not decidable.
But Theorem \ref{CayleyDecidable} prohibits the existence of such groups.
This observation calls for finding more sophisticated methods for proving non graph automaticity of groups.  Below we provide one such simple method.
\smallskip

The next lemma puts a significant restriction on functions in automatic structures.

\begin{lemma}[Constant Growth Lemma] \label{L:CGL}  Let $f:D^n\rightarrow D$ be a function on $D\subseteq \Sigma^\star$ such that the graph of $f$ is FA recognizable. There exists a constant $C$ such that for all $x_1,\ldots, x_n\in D$ we have
$$
|f(x_1,\ldots, x_n)| \leq max\{|x_i| \mid i=1,\ldots, n\}+C.
$$
\end{lemma}
\begin{proof}
Let $C_1$, $C_2$ be the number of  states of  finite automata recognizing the graph of the function $f$ and the domain $D$, respectively.
Assume that there exist $x_1,\ldots, x_n\in D$ such that
$$
|f(x_1,\ldots, x_n)| > max\{|x_i| \mid i=1,\ldots, n\}+C_1\cdot C_2.
$$
Let $y=f(x_1,\ldots, x_n)$. We can write $y$ as
$$
y=\con(x_1,\ldots, x_n,z) \cdot \con (\lambda, \ldots, \lambda,z')
$$
where $|z|=max\{|x_i| \mid i=1,\ldots, n\}$. Since $\M$ accepts $\con(x_1,\ldots, x_n,y)$, by the Pumping Lemma $z$ can be pumped to a longer string $u\in \Sigma^\star$ such that $u\neq z$, $\con(x_1,\ldots, x_n,u)$ is accepted by $\M$ and $u\in D$. But $f$ is a function $D$. The desired
$C$ is $C_1\cdot C_2$.
\end{proof}

Let $\A$ be an automatic structure with atomic functions  $f_1$, $\ldots$, $f_m$. For a set $E=\{e_1,\ldots, e_k\}$ of elements of the structure define the following sequence by induction:
$$
G_1(E)=E, \ G_{n+1}(E)=G_n(E) \cup \{f_i(\bar{a}) \mid \bar{a} \in G_n(E), i=1,\ldots,m\}, \ n>0.
$$
\begin{theorem}[Growth of Generation Theorem] In the setting above, there exists a constant $C$ such that $|a|\leq C\cdot n$ for all  $a\in G_n(E)$. Hence, for all $n\geq 1$
$$
|G_n(E)| \leq |\Sigma|^{C\cdot n} \ \ \ \mbox{if $|\Sigma|>1$}.
$$
If $|\Sigma|=1$ then $|G_n(E)|\leq C\cdot n$.
\end{theorem}

\begin{proof}
Let $C_i$ be the constant stated in the previous lemma for the function $f_i$. Let $C'=max\{C_1,\ldots, C_m, |e_1|, \ldots, |e_k|\}$. By the lemma above, using induction on $n$ one can easily prove that for all $a\in G_n(E)$ we have $|a|\leq (C'+1) \cdot n$.
Set $C=C'+1$. Now we clearly have $G_n(E) \subseteq \Sigma^{\leq C\cdot n}$. Hence the theorem is proved.
\end{proof}

The result above can be applied to provide many examples of non automatic structures.
For instance,  the following structures do not have automatic presentations:
\begin{itemize}
\item The semigroup $(\Sigma^\star; \cdot)$.
\item The structure $(\omega; f)$, where $f:\omega^2 \rightarrow \omega$ is a pairing function (that is a bijection between $\omega^2$ and $\omega$).
\item The free group $F(n)$ with $n$ generators, where $n>1$.
\item The structures $(\omega; Div)$ and $(\omega; \times)$. \qed
\end{itemize}
For reference see \cite{KN95} \cite{BG00}.

\smallskip

For groups, the theorem above implies the following corollary:




\begin{corollary}
Finitely generated groups whose word problem can not be solved in quadratic time are not graph automatic. \qed
\end{corollary}

Of course, there are groups whose word problems can not be decided in  quadratic time. For instance, see \cite{Sapir2010}. However, the authors do not know of any natural example of such groups.



\section{Finitely generated FA presentable groups}
\label{se:FAgroups}

The Definition \ref{CentralDefinition} suggests that automaticity into groups can also be introduced
by requiring that the group operation is FA recognizable. In this section we do exactly this by  considering
groups $(G, \cdot)$ in which the group operation $\cdot$ is automatic. We recast the definition:

\begin{definition}
We say that a group $G$ is {\bf FA presentable} if  the following conditions are satisfied:
\begin{itemize}
\item The domain of $G$ is FA recognizable set.
\item The graph of the group operation, that is, the set $\{(u,v,w) \mid u\cdot v=w\}$
is FA recognizable.
\end{itemize}
\end{definition}

Note that the definition does not require that $G$ is finitely generated. Examples of
FA prsentable  groups are the following:
\begin{itemize}
\item The additive group of $p$-adic rational numbers: $\Z[1/p]$.
\item Finitely generated Abelian groups.
\item The infinite direct sum $\bigoplus G$ of a finite group $G$.
\end{itemize}
For abelian (not necessarily finitely generated) FA-presentable groups see \cite{Nies07} \cite{NS09} \cite{Ts09}.
We also mention a recent result of Tsankov that the additive group of rational numbers is not FA presentable.
The proof uses advanced techniques of additive combinatorics \cite{Ts09}.
\smallskip

Our goal is to give a full characterization of finitely generated FA-presentable groups.
Our proof follows Thomas and Oliver \cite{RickOliver05}. We start with the following definition.

\begin{definition}
An infinite group is {\bf virtually Abelian} if it has torsion free normal Abelian subgroup
of finite index.
\end{definition}

An example of virtually Abelian group is $D_{\omega}$, the infinite dihedral group.
One can view this group as the automorphism group of the graph that looks like the
bi-infinite chain.

\begin{lemma} \label{L:GR-Groups}
Finitely generated virtually Abelian groups all are FA presentable.
\end{lemma}

\begin{proof}
Let $G$ be a finitely generated virtually Abelian group. By the definition, there exists an Abelian torsion free normal subgroup $A$ of $G$ which has a finite index, say $n$,
in $\G$.
We can assume that $\A$ is isomorphic to $\Z^k$. Let
$x_1$, $\ldots$, $x_k$ be the generators of $\Z^k$.
Without loss of generality, to avoid notations, we assume that $k=2$.

\smallskip

Let $t_1$, $\ldots$, $t_n$ be all representatives of the quotient group  $G/\Z^k$.
Every element $g\in G$ can be written as \
$t_i x_1^{a_1} x_2^{a_2}$ \ where  $a_1, a_2 \in Z$ and $i\in \{1,\ldots, n\}$.
Since $\Z^k$ is normal, we also have the following list of equalities for some fixed integers $c_{1,1,j}$, $c_{1,2,j}$, $c_{2,1,j}$ and $c_{2,2,j}$:
\begin{center}
{$x_1 t_j=t_j x_1^{c_{1,1,j}} x_2^{c_{1,2,j}}$ \  and \  $x_2 t_j=t_j x_2^{c_{2,1,j}} x_2^{c_{2,2,j}}$ where $j=1,\ldots, n$.}
\end{center}
In addition, there are integer constants $c_i$ and $c_j$ such that $t_i t_j= t_k x^{c_i}_1 x_2^{c_j}$ for all $i,j=1,\ldots, n$. Taking all these into account we can now perform the group operation on $G$ as follows:
\begin{center}{
$t_i x_1^{a_1} x_2^{a_2} \cdot t_j x_1^{b_1} x_2^{b_2}=t_it_j
x_1^{a_1c_{1,1,j}+a_2c_{2,1,j}+b_1}  x_2^{a_1c_{1,2,j}+a_2c_{2,2,j}+b_2}$=
$t_k x^{c_i}_1 x_2^{c_j}x_1^{a_1c_{1,1,j}+a_2c_{2,1,j}+b_1}  x_2^{a_1c_{1,2,j}+a_2c_{2,2,j}+b_2}$.
}
\end{center}
All these operations can now be performed by finite automata.  This proves the lemma.
\end{proof}

The next lemma again suits a more general case of monoids. A {\bf monoid} is a structure $(M; \cdot)$, where $\cdot$ is an associative binary operation on $M$.

\begin{lemma} \label{L:Monoids}
If $(M;\cdot)$ is an automatic monoid then for all $m_1,\ldots, m_n \in M$ the following inequality holds true:
$$
|m_1 \cdots \ldots \cdot m_n| \leq max\{|m_i| \mid i=1,\ldots, n\} + C\cdot log(n),
$$
where $C$ is a constant.
\end{lemma}

\begin{proof}
Let $C$ be the the number required by The Constant Growth Lemma (see Lemma \ref{L:CGL}). By the lemma we have:
$$
(\#) \hspace{5mm} |m_1\cdot m_2| \leq max\{|m_1|, |m_2|\}+C
$$
for all $m_1, m_2\in M$. So, for $n=1,2$ the lemma is obvious. For $n>2$ we write $n=n_1+n_2$ with  $n_1=n/2$.  Consider the elements
$$
x=m_1\cdot \ldots \cdot m_{n_1} \hspace{1mm} and \hspace{1mm} y=m_{n_1+1}\cdot \ldots \cdot m_n.
$$
From the induction assumption we have the following inequalities:
$$
|x| \leq max_{1\leq i\leq  n_1} |m_i| + C\cdot log(n_1) \hspace{2mm} \mbox{and} \hspace{2mm}
|y|\leq max_{n_1< i \leq  n} |m_i| + C\cdot log(n_1)
$$
Therefore from the inductive assumption and $(\#)$ we have:
$$
|x\cdot y| \leq max\{|x|, |y|\}+C \leq max\{|m_i| \mid i=1,\ldots, n\} + C\cdot log(n).
$$
Thus, we have the desired inequality.
\end{proof}


Let $X=\{g_1,\ldots, g_k\}$ be the set of generators of the group $G$. For each element $g\in G$, let $\delta(g)$ be the minimum  $n$ such that $g=a_1\cdot \ldots \cdot a_n$ in the group $G$, where each $a_i \in X$. Now we define the following:
$$
G_n=\{g\in G\mid \delta(g)\leq n\} \ \mbox{and} \ gr_{\G}(n)=|G_n|.
$$
The function $gr_{\G}$ is called the {\bf growth function} of the group $G$.

\begin{lemma}
If $G$ is FA-presentable then its growth function is bounded by a polynomial.
\end{lemma}

\begin{proof}
By Lemma \ref{L:Monoids}, for each $g\in G_n$ we have \
$\delta(g)\leq C log(n)$,  where $C$ is a constant.
Therefore there is a constant $C_1$ such that
$$
(\star) \hspace{15mm} gr_{\G}(n)\leq |\Sigma|^{Clog(n)}\leq 2^{C_1log(n)}\leq n^{C_1}.
$$
This proves the lemma.
\end{proof}

Now we need two deep results  from group theory. The first  is the theorem of Gromov stating that  finitely generated groups with polynomial growth  are all virtually nilpotent \cite{gromov}. The second  is  the theorems of Romanovski \cite{Romanovski80} and Noskov\cite{Noskov84} stating that a virtually solvable group has a decidable first order theory if and only if it is virtually Abelian. Virtually nilpotent groups are virtually solvable.
Thus,we have proved the following:

\begin{theorem} \label{FA-presentable}
A finitely generated group is FA-presentable if and only if it is virtually Abelian.
\end{theorem}

Since all virtually abelian groups are graph automatic we have the following result:

\begin{corollary}
Every finitely generated FA-prsentable group is graph automatic.\qed
\end{corollary}

\end{document}